\documentclass[reqno]{amsart}

\usepackage[fontsize=11pt]{scrextend}
\usepackage[T1]{fontenc}
\usepackage{pgf,pgfarrows,pgfnodes,pgfautomata,pgfheaps,pgfshade,hyperref, amssymb}
\usepackage{amssymb}

\usepackage[english]{babel}
\usepackage[capitalize]{cleveref}
\usepackage{mathtools}
\usepackage[colorinlistoftodos]{todonotes}
\usepackage{soul}
\usepackage[utf8]{inputenc} 
\usepackage[T1]{fontenc}
\usepackage[normalem]{ulem}

\usepackage{tikz}
\usetikzlibrary{arrows}

\usepackage{xcolor}

\usepackage{faktor}
\usepackage{xfrac}

\usepackage{enumerate}
\usepackage{comment}

\newcommand{\norm}[1]{\left\Vert #1\right\Vert}

\def\Z{{\mathbb Z}}
\def\C{{\mathbb C}}
\def\Q{{\mathbb Q}}
\def\N{{\mathbb N}}

\newcommand*\diff{\mathop{}\!\mathrm{d}}

\def\cF{{\mathcal F}}

\hypersetup{
    colorlinks,
    linkcolor={blue!30!black},
    citecolor={green!50!black},
    urlcolor={blue!80!black}
} 

\usepackage{mathrsfs} 
\usepackage{dsfont}

\usepackage{graphicx} 

\newtheorem{theorem}{Theorem}[section]
\newtheorem{proposition}[theorem]{Proposition}

\newtheorem{lemma}[theorem]{Lemma}

\newcommand{\supp}{\operatorname{supp}}

\newcounter{thmcounter}

\newcounter{introthmcounter}

\newtheorem{corollary}[theorem]{Corollary}
\theoremstyle{definition}
\newtheorem{definition}[theorem]{Definition}
\newtheorem*{definition*}{Definition}
\newtheorem{question}[theorem]{Question}
\newtheorem*{question*}{Question}
\newtheorem*{remark*}{Remark}
\newtheorem*{lemma*}{Lemma 2.8}
\newtheorem*{lemma**}{Lemma 2.9}
\newcounter{proofcount}
\AtBeginEnvironment{proof}{\stepcounter{proofcount}}

\makeatletter
\def\section{\@startsection{section}{1}%
  \z@{.5\linespacing\@plus.7\linespacing}
{.8\baselineskip}%
  {\normalfont\fontsize{14}{14}\centering\bfseries}%
}

\makeatletter
\def\subsection{\@startsection{subsection}{2}%
  \z@{.4\linespacing\@plus.7\linespacing}
{.6\baselineskip}%
  {\normalfont\fontsize{12}{13}\centering\bfseries}%
}

\makeatletter                  
\@addtoreset{claim}{proofcount}
\makeatother

\theoremstyle{remark}

\newtheorem{remark}[theorem]{Remark}

\def\N{{\mathbb N}}

\newcommand{\xmt}{(X,\mu,T)}

\newcommand{\I}{\mathcal{I}}

\newcommand{\gen}{\texttt{\textup{gen}}}

\newcommand{\Erdos}{Erd\H{o}s}
\newcommand{\Folner}{F\o{}lner}

\title{Asymmetric infinite sumsets in large sets of integers}
\author{Ioannis Kousek}
\address{Department of mathematics, University of Warwick}
\email{ioannis.kousek@warwick.ac.uk}
\date{}

\begin{document}

\begin{abstract}
We show that for any set $A\subset \N$ with 
positive upper density and any $\ell,m \in \N$, there exist an infinite 
set $B\subset \N$ and some $t\in \N$ so that $\{mb_1 + \ell b_2 \colon 
b_1,b_2\in B\ \text{and}\ b_1<b_2 \}+t \subset A,$ verifying a conjecture of Kra, Moreira, Richter and Robertson. 
We also consider the patterns $\{mb_1 + \ell b_2 \colon 
b_1,b_2\in B\ \text{and}\ b_1 \leq b_2 \}$, for infinite $B\subset \N$ and prove that any set $A\subset \N$ with lower density $\underline{\diff}(A)>1/2$ contains such configurations up to a shift. We show that the value $1/2$ is optimal and obtain analogous results for values of upper density and when no shift is allowed. 
\end{abstract}

\maketitle

\section{Introduction}

In \cite{kmrr2}, Kra, Moreira, Richter and 
Robertson established -- among other things -- the following 
result, resolving a well-known conjecture of Erd\H{o}s.

\begin{theorem}\label{KMRR theorem}  \cite[Theorem 1.2]{kmrr2}
For any $A\subset \N$ with positive upper Banach density there exists some infinite set $B\subset \N$ and a number $t\in \N$ such that 
$$\{b_1 + b_2\colon  b_1,b_2\in B,\ b_1 \neq b_2\} + t \subset 
A.$$    
\end{theorem}
For completeness, we recall that for a set $A\subset \N$ its upper Banach density, denoted by $\diff^{*}(A)$, is defined as the limit $$\diff^{*}(A)=\limsup_{N-M \to \infty} \frac{|A\cap \{M,M+1,\dots,N\}|}{N-M}.$$
More recently, the same set 
of authors proposed a conjecture (see 
\cite[Conjecture 3.10]{kmrr_survey}) which
generalises \cref{KMRR theorem}. Our first main result 
verifies this conjecture 
and is the following.

\begin{theorem} \label{mB+lB}
For any $A\subset \N$ with positive upper 
Banach density and $\ell,m \in \N$, there exists
some infinite set $B\subset \N$ and a number 
$t\in \N$ such that 
$$\{mb_1 + \ell b_2\colon  b_1,b_2\in B,\ b_1 < b_2\} +t \subset 
A.$$            
\end{theorem}

\begin{remark} \label{unrestricted patterns counterexample}
If $\ell=m=1$, Theorem \ref{mB+lB} coincides with \cref{KMRR theorem} and more generally, when $\ell=m$, Theorem \ref{mB+lB} can easily be deduced from \cref{KMRR theorem}. If $\ell \neq m$, shifts of the patterns $\{mb_1 + \ell b_2\colon  b_1,b_2\in B,\ b_1 \neq b_2\}$ can not always be found in sets 
of positive density. In fact, these sumsets are not 
even partition regular as shown in \cite[Example $3.9$]{kmrr_survey}.  
\end{remark}

Our proof of \cref{mB+lB} is ergodic theoretic in nature. The 
main setup for it is laid out in Section \ref{setup} and it 
is completed in Section \ref{main proofs}, along with a slightly stronger result (see \cref{lB in shift}).

For the purposes of contextualizing our next main results, we redirect
our attention to Theorem \ref{KMRR theorem}. In 
particular, we point out that the restriction $b_1\neq b_2$ is 
necessary and it was long known that there exist sets of full upper 
Banach density not containing infinite sumsets 
$\{b_1+b_2: b_1,b_2 \in B\}$ up to shifts. A natural question then is whether one can guarantee such unrestricted sumsets in sets which are large through stronger notions of density. Recall that for a set $A\subset \N$, its natural upper and lower densities, denoted by $\overline{\diff}(A)$ and $\underline{\diff}(A)$, respectively, are defined as the limits 
$$\overline{\diff}(A)=\limsup_{N\to \infty} \frac{|A \cap \{1,\ldots,N\}}{N}\ \text{and}\ \underline{\diff}(A)= \liminf_{N\to \infty} \frac{|A \cap \{1,\ldots,N\}|}{N}.$$
In \cite{kousek_radic2024}, the author and Radi\'c gave a solution to the unrestricted version of this problem for natural upper and lower density, via the following result.

\begin{theorem} \label{KousekRadic theorem}  \cite[Theorems 1.2, 1.3]{kousek_radic2024}
 Let $A \subset \N$. 
    \begin{enumerate}
    \item If $\overline{\diff}(A)>5/6$ or $\underline{\diff}(A)>3/4$, there exists an infinite set $B\subset \N$ such that $B+B \subset A$.
    \item If $\overline{\diff}(A)>2/3$ or $\underline{\diff}(A)>1/2$, there exist an infinite set $B\subset \N$ and $t\in \{0,1\}$ such that $B+B+t \subset A$.
\end{enumerate}    
\end{theorem}

\begin{remark} \label{optimality remark}
It was also shown in \cite{kousek_radic2024} that both of these
results are optimal in the sense that, for example, there exists $A\subset \N$ with $\overline{\diff}(A)=5/6$, such that $B+B\not\subset A$ for any infinite $B\subset \N$.      
\end{remark}

Analogously, we are also interested in an unrestricted 
version of Theorem \ref{mB+lB}. Including 
the diagonal in the sumsets is an important first step in this direction. Building on the ideas involved in the proof 
of 
Theorem \ref{KousekRadic theorem} and our proof of Theorem 
\ref{mB+lB}, we are able to prove the following results.  

\begin{theorem}\label{lower density theorem intro}
Let $\ell,m \in \N$. For any $A\subset \N$ with $\underline{\diff}(A)>1/2$, there exist an infinite set $B\subset \N$ and some $t\in \N$ such that 
$$\{mb_1 + \ell b_2\colon  b_1,b_2\in B\ \text{and}\ b_1 \leq b_2\} + t \subset A.$$
\end{theorem}

\begin{theorem} \label{mB+lB unrestricted shift}
Let $\ell,m \in \N$ and $k=m/\ell$. For any $A\subset \N$ with $\overline{\diff}(A)>1-1/(k+2)$, there exist an infinite set $B\subset \N$ and some $t\in \N$ such that 
$$\{mb_1 + \ell b_2 : b_1,b_2 \in B\ \text{and}\ b_1 \leq b_2\} + t \subset A.$$
\end{theorem}

We will in fact show that both the previous bounds are optimal in the sense of Remark \ref{optimality remark}. We also stress that the bounds established in \cref{lower density theorem intro} do not depend on the parameters $m,\ell$, unlike the bounds in \cref{mB+lB unrestricted shift}, which do so implicitly, as they depend on the ratio $m/\ell$. On another note, it is easy to see that the shift $t$ in Theorems \ref{mB+lB}, \ref{lower density theorem intro} and \ref{mB+lB unrestricted shift} can be chosen from $\{0,1,\ldots,\ell+m-1\}$. Indeed, write $t=(\ell+m)j+i$, for some $j\in \N_0$ and $i\in \{0,1,\ldots,\ell+m-1\}$. Then, for example, the inclusion $\{mb_1 + \ell b_2 : b_1,b_2 \in B\ \text{and}\ b_1 \leq b_2\} + t \subset A$ can be rewritten as $\{mb_1 + \ell b_2 : b_1,b_2 \in (B+j)\ \text{and}\ b_1 \leq b_2\} + i \subset A$.   

We also prove similar results for the case of unshifted patterns.

\begin{theorem} \label{mB+lB unrestricted no shift}
Let $\ell,m \in \N$ and $k=m/\ell.$ For any
$A\subset \N$ with $\overline{\diff}(A)>1-1/\left(\ell(k+1)(k+2)\right)$, there 
exists an infinite set $B\subset \N$ such that 
$$\{mb_1 + \ell b_2 : b_1,b_2 \in B\ \text{and}\ b_1 \leq b_2\} \subset A.$$
\end{theorem}
In Section \ref{counterexamples} we prove that Theorems \ref{mB+lB unrestricted shift} and \ref{mB+lB unrestricted no shift} are optimal in the sense of Remark
\ref{optimality remark}. An analogue of \cref{mB+lB unrestricted no shift} with lower density threshold of $\underline{\diff}(A)>1-1/(2(\ell+m))$ is proven in Section \ref{lower density section}, where we also show that this and Theorem \ref{lower density theorem intro} are optimal. For a discussion about other potential unrestricted versions of Theorem \ref{mB+lB} we refer the reader to Section \ref{questions}.

Our proofs of the above results use ergodic theory. 
Expanding on the ideas introduced in \cite{kmrr2}, given a set $A\subset \N$, we first relate the inclusion $\{mb_1 + \ell b_2\colon  b_1,b_2\in B\ \text{and}\ b_1 < b_2\} + t \subset A$ to the existence of a specific system $(X,\mu, T)$, a triple $(a,x_1,x_2) \in X^3$, with predetermined $a\in X$, such that $\left( T^{\ell} \times T^m \right) ^{n_i} \left( a, x_1 \right) \to (x_1,x_2)$, along a sequence $(n_i)_{i \in \N}$. In Section \ref{s dynamics to combinatorics 1.2} we explain how a classical version of Furstenberg's correspondence principle allows us to translate Theorem \ref{mB+lB} to a dynamical statement of the above form, that is, \cref{dynamical mB+lB}. 

The problem gets more complicated -- already at the
level of the correspondence principle -- if we also 
want to include the diagonal in the sumsets. More 
precisely, in the above dynamical setting, in order 
to guarantee that $\{ mb_1 + \ell b_2\colon  b_1,b_2\in B\ \text{and}\ b_1 \leq b_2\}+t \subset A$, 
we additionally need to know that $(T^{j \ell})^{n_i}a \to x_2$, for some $j\in \N$ such that $j\ell=\ell+m$. However, this equation is only solvable if $k=m/\ell$ is an integer. In this special case, one could devise a modified version of the correspondence principle, by building an appropriated 
$(T^{(k+1)}\times T)$-invariant probability measure in 
$(\{0,1\}^{\Z} \times \{0,1\}^{\Z}, T^{(k+1)} \times T)$. For the case $\ell=m=1$, this argument was utilised in \cite{kousek_radic2024}. 

To handle the general case when the ratio $k=m/\ell$ is not necessarily an integer, we consider a 
$(T^{(\lceil k \rceil +1)}\times T)$-invariant probability measure in 
$(\{0,1\}^{\Z} \times \{0,1\}^{\Z}, T^{(\lceil k \rceil +1)} \times T)$, arising from a generic pair of points $(a',a)$, where $a$ corresponds to the indicator of $A$, and $a'$ corresponds to the indicator of an auxiliary set $A'\subset \N$. Essentially, $A'$ is such that the inclusion $(\ell(\lceil k \rceil +1)B+t \subset A'$ also implies 
$(m+\ell)B + t \subset A$. In Section \ref{unrestricted results} we formulate this version of the correspondence principle as Lemma \ref{FCP shift general k} and use it in order to deduce Theorem \ref{mB+lB unrestricted shift} from a dynamical statement, namely Theorem \ref{dynamical mB+lB unrestricted shift}. The proof of Lemma \ref{FCP shift general k}, along with that of another correspondence principle which is used for the case of unshifted patterns $\{ mb_1 + \ell b_2\colon  b_1,b_2\in B\ \text{and}\ b_1 \leq b_2\}$, is given in Section \ref{FCP section for k in Q}. 

One cannot overestimate the influence of the 
pioneering work presented in \cite{kmrr2} on recent results 
pertaining to the ergodic theory approach to infinite 
sumsets. For work related to this interesting and flourishing theory see
\cite{ackelsberg2024counterexamples}, 
\cite{charamaras_mountakis2024}, 
\cite{DNGJLLM}, \cite{host19}, \cite{kmrr1}, \cite{mrr19} (other 
infinite sumset results via different methods can be
found in \cite{Granville},\cite{Maynard},\cite{TaoZieg}). 

Noteworthily, it was privately communicated to the author that
Felipe Hernández has independently found a (different) proof 
of Theorem \ref{mB+lB}, which has not yet been published.

\medskip

\textbf{Acknowledgments.}
The author is thankful to Joel Moreira, Nikos Frantzikinakis, 
Felipe Hernández, Trist\'an Radi\'c, Vicente Saavedra-Araya and Andreas Mountakis for helpful comments on an earlier draft of this article. 

\section{Translation to dynamics}\label{S dynamics to combinatorics}

\subsection{Restricted sumsets} \label{s dynamics to combinatorics 1.2}

For the reader's convenience we recall some 
standard concepts. A \emph{topological system} 
is a 
pair $(X,T)$, where $X$ is a compact metric 
space and $T\colon X \to X$ a homeomorphism. 
Whenever there is a $T$-invariant Borel 
probability measure $\mu$ on $X$, we call 
$(X,\mu,T)$ a \emph{measure preserving system}.  

The system $(X,\mu,T)$ is \emph{ergodic} if the only 
$T$-invariant sets have either measure $0$ or $1$. We 
denote the \emph{support} of the measure $\mu$, 
which is defined as the smallest closed subset of 
$X$ with full measure, by $\supp(\mu)$. 

Given a measure preserving system $\xmt$, a function $f\in L^2(X)$ is called weak-mixing if 
$$\lim_{N\to \infty} \frac{1}{N} \sum_{n=1}^N \left| \int_X T^nf \cdot \bar{f}\ d\mu \right| =0.$$

A \emph{F\o lner sequence} $\Phi$ in $\N$ is a sequence of (non-empty) finite sets $N \mapsto \Phi_N \subset \N$, $N\in \N$ such that
$$\lim_{N\to \infty} \frac{\left| \Phi_N \cap \left(t+\Phi_N\right) \right|}{|\Phi_N|}=1,$$
for any $t\in \N$. Given a system $(X,\mu,T)$, a point $a\in X$ is $T$-\emph{generic} for $\mu$ 
along a F\o lner sequence $\Phi$, written as $a\in \textbf{gen}(\mu, T, \Phi)$, if 
$$\mu = \lim_{N\to \infty} \frac{1}{|\Phi_N|} \sum_{n\in \Phi_N } \delta_{T^n a},$$
where $\delta_x$ is the Dirac mass at $x\in X$ and the limit is in the weak* topology. 

We next define the concept of dynamical 
progressions -- which parallels that of Erd\H{o}s 
progressions introduced in \cite{kmrr2} -- which as 
we shall see are connected to the combinatorial 
patterns that we are looking for in Theorem \ref{mB+lB}.

\begin{definition} \label{EP def}
Given a topological system $(X,T)$ and natural numbers $\ell,m$, we say that a point $(x_0,x_1,x_3)\in X^3$ is an $(\ell,m)$-Erd\H{o}s progression if there exists a sequence $n_1<n_2<\cdots$ of integers such that $(T^{\ell} \times T^m)^{n_i}(x_0,x_1) \xrightarrow{} (x_1,x_2)$ as $i\to \infty$.    
\end{definition}

Our first main dynamical result, the one behind \cref{mB+lB}, is the 
following theorem. 

\begin{theorem}\label{dynamical mB+lB}
Let $(X,\mu,T)$ be an ergodic system, let $a\in \textbf{gen}(\mu,T,\Phi)$ for some F\o lner 
sequence $\Phi$ and $E \subset X$ be an open set with $\mu(E)>0$. Then, there exist $x_1,x_2 \in X$ and $t\in \N$ so that $(a,x_1,x_2) \in X^3$ is an $(\ell,m)$-Erd\H{o}s progression and $T^tx_2 \in E$.    
\end{theorem}

It turns out that Theorems \ref{mB+lB} and 
\ref{dynamical mB+lB} are actually equivalent, 
but we shall only deal with the required direction here. 
For this we need the following -- parallel of \cite[Theorem $2.2$]{kmrr2} -- result.

\begin{proposition}\label{EP and return times}
Fix a topological system $(X,T)$ and open sets $U,V \subset X$. If there exists an $(\ell,m)$-Erd\H{o}s progression $(x_0,x_1,x_2)\in X^3$ with $x_1\in U$ and $x_2 \in V$, then there exists some infinite set $B\subset \{n\in \N\colon T^{\ell n}x_0 \in U\}$ such that $\{mb_1+\ell b_2\colon b_1,b_2\in B,\ b_1<b_2\} \subset \{n\in \N\colon T^nx_0 \in V\}$.    
\end{proposition}

\begin{proof}
By the definition of $(\ell,m)$-\Erdos\ progressions 
we can find a strictly 
increasing sequence $(c_n)_{n\in \N}$ such that $(T^{\ell} \times T^m)^{c_n}(x_0,x_1) \xrightarrow{} (x_1,x_2)$ as $n\to \infty$ and each 
$c_n$ is such that $T^{\ell c_n}x_0 \in U$. We
will construct $B \subset \{c_n: n\in \N\}$ 
inductively and, the basis of the induction being 
simple, we only prove the inductive step. Suppose 
$b_1<\dots < b_n$ have been chosen so that 
$$x_0 \in \bigcap_{1\leq i < j \leq n} T^{-mb_i -\ell b_j}V\ \text{and}\ x_1\in \bigcap_{1\leq i \leq n} T^{-m b_i}V.$$
Then, we can choose $b_{n+1} \in \{c_k: k \in \N\}$ with $b_{n+1}>b_n$ and such that 
$$(T^{\ell} \times T^m)^{b_{n+1}}(x_0,x_1) \in \left( \bigcap_{1\leq i \leq n} T^{-m b_i}V \right) \times V.$$
It follows that 
$$x_0 \in \bigcap_{1\leq i < j \leq n+1} T^{-mb_i -\ell b_j}V\ \text{and}\ x_1\in \bigcap_{1\leq i \leq n+1} T^{-m b_i}V$$
and this concludes the induction. We finish the 
proof by letting $B=\{b_n: n\in \N\}$.
\end{proof}

To prove that \cref{dynamical mB+lB} implies \cref{mB+lB} we shall use the following classical version of Furstenberg's correspondence principle. 

\begin{lemma}\label{FCP}\cite[Theorem 2.10]{kmrr1}
For a set $A\subset \N$ with $\diff^{*}(A)>0$ 
there exists an ergodic system $\xmt$, a \Folner\ 
sequence $\Phi$, a point $a\in \gen(\mu,T,\Phi)$ and a 
clopen set $E\subset X$ such that $\mu(E)>0$ and $A=\{n\in \N: T^na \in E\}$.
\end{lemma}

\begin{proof}[Proof that \cref{dynamical mB+lB} 
implies \cref{mB+lB}]
Let $A \subset \N$ with $\diff^{*}(A)>0$ and 
$\xmt$, $a\in X$, $\Phi$ and $E\subset X$ be 
those 
arising from Lemma \ref{FCP}. By Theorem \ref{dynamical mB+lB}, there exists a $t\in \N$ and an $(\ell, m)$-Erd\H{o}s progression $(a,x_1,x_2) \in \{a\} \times X \times T^{-t}E$. Invoking Proposition \ref{EP and return times} we obtain an infinite set $B\subset \N$, such that $\{mb_1+\ell b_2\colon b_1,b_2\in B,\ b_1<b_2\} \subset \{n\in \N\colon T^na \in T^{-t}E\}$. Since $A=\{n\in \N: T^na \in E\}$, we see that $A-t=\{n\in \N\colon T^na \in T^{-t}E\}$, so the theorem follows.
\end{proof}

Apropos of this discussion, we address the necessity of the shift 
in Theorem \ref{mB+lB} and also the density threshold for the 
unshifted version. This is merely an observation, but for the 
reader's convenience we prove it in the next proposition. 

\begin{proposition} \label{mB+lB no shift restricted}
Let $\ell,m\in \N$. If $A\subset \N$ with $\diff^{*}(A)>1-\frac{1}{(\ell+m)}$, then there is an infinite set $B\subset \N$ such that
$$\{mb_1 + \ell b_2\colon  b_1,b_2\in B,\ b_1 < b_2\} \subset 
A.$$
Otherwise, the shift in Theorem \ref{mB+lB} is in general 
necessary. 
\end{proposition}

\begin{proof}
Observe that the set $A=\N \setminus (\ell+m)\N$ has natural 
density $1-\frac{1}{(\ell+m)}$ and contains no infinite sumset of 
the form 
$\{mb_1 + \ell b_2\colon  b_1,b_2\in B,\ b_1 < b_2\}$. Indeed, the infinity of $B$ allows us 
to choose an infinite subset of it, say $B'\subset B$, all the 
elements of which are equal modulo $(\ell+m)$. That is, there is 
some $j\in \{0,1,\ldots,\ell+m-1\}$ so that any $b\in B'$ is of 
the form $b=(\ell+m)n+j$, some $n\in \N$. It follows that $mb_1+\ell b_2 \in (\ell+m)\N$, for any $b_1,b_2\in B'$. This 
means that the shift above is necessary and the density threshold 
cannot be improved. 

On the other hand, if $A\subset \N$ with $\diff^{*}(A)>1-\frac{1}{(\ell+m)}$, it is easy to see that $\diff^{*}(A\cap (\ell+m)\N)>0$ and then by \cref{mB+lB} there is some infinite $B\subset \N$ such that 
$$\{mb_1 + \ell b_2\colon  b_1,b_2\in B,\ b_1 < b_2\} \subset A\cap (\ell+m)\N \subset A.$$ 
Again, for the latter we implicitly used the fact that infinity 
of $B$ allows us 
to choose an infinite subset of it, all the elements of 
which are equal modulo $(\ell+m)$.   
\end{proof}

\subsection{Lifting restrictions} \label{unrestricted results}

In order to prove the combinatorial results 
in Theorems \ref{mB+lB unrestricted shift} and \ref{mB+lB unrestricted no shift} we will use the following dynamical results respectively.

\begin{theorem}\label{dynamical mB+lB unrestricted shift}
Let $(X,\mu,T)$ be an ergodic system and $a\in \textbf{gen}(\mu,T,\Phi)$ for some F\o lner 
sequence $\Phi$. Moreover, let $\ell,m\in \N$, $q=\lceil m/\ell \rceil$ and assume that 
$E_1,\ldots,E_{\ell+m},F_1,\ldots,F_{\ell+m}\subset X$ are open sets 
such that $F_j=T^{-(j-1)}F_1$, $j=1,\ldots,\ell+m$, 
$E_{i+(q+1)}=T^{-1}E_i$, for $i=1,\ldots,\ell+m-q-1$, and also  
\begin{equation} \label{mB+lB shift eq}
(\ell+m)\mu(F_1)+\ell \left(\mu(E_1)+\cdots + \mu(E_{q+1}) \right)>\ell(q+1).     
\end{equation}
Then, for some $j\in \{1,\ldots,\ell+m\}$, there exist $x_1,x_2 \in X$ so that $(a,x_1,x_2) \in X^3$ is an $(\ell,m)$-Erd\H{o}s progression 
and $(x_1,x_2) \in E_j \times F_j$.   
\end{theorem}

\begin{theorem}\label{dynamical mB+lB unrestricted no shift}
Let $(X,\mu,T)$ be an ergodic system, let $a\in \textbf{gen}(\mu,T,\Phi)$ for some F\o lner 
sequence $\Phi$ and $E,F \subset X$ be an open sets with
\begin{equation}\label{mB+lB no shift eq}
(\ell+m)\mu(F)+\ell\mu(E)>2\ell+m-1,   
\end{equation} 
for some $\ell,m \in \N$. Then, there exist $x_1,x_2 \in X$ such that $(a,x_1,x_2) \in X^3$ is an $(\ell,m)$-Erd\H{o}s progression and also $(x_1,x_2) \in E \times F$.  
\end{theorem}

To facilitate the transition from ergodic theory to
combinatorics in this setting, we shall again utilise the notion of 
$(\ell,m)-$\Erdos\ progressions as in
Definition \ref{EP def}, as well as Proposition
\ref{EP and return times}. However, the previously used, more
classical version of Furstenberg's correspondence principle 
seems to no longer be
useful and we need the adaptations presented in Lemmas \ref{FCP shift general k} and \ref{FCP no shift general k} below, for Theorems \ref{mB+lB unrestricted shift} and \ref{mB+lB unrestricted no shift}, respectively. 

Before stating the lemmas, we establish some notation; $\Sigma$ denotes the space $\{0,1\}^{\Z}$ and is endowed with
the product topology so that it is compact metrizable. 
We also let $S \colon \Sigma \to \Sigma$ denote the shift transformation given by $S(x(n))=x(n+1)$, for any $n\in \Z$, $x=(x(n))_{n\in \Z} \in \Sigma$.

\begin{lemma} \label{FCP shift general k}
 Let $A\subset \N$ and $\ell,m \in \N$ with $k=m/\ell$ and let $q=\lceil k \rceil$, the ceiling of $k$. 
 Then, there exist an ergodic system $(\Sigma 
 \times \Sigma, \mu, S^{(q+1)} \times S)$, an open set $E\subset \Sigma$, a pair of points $a,a'\in \Sigma$ and a F\o lner sequence 
 $\Phi$, such that $(a',a)\in \textbf{gen}(\mu,\Phi)$ and 
 \begin{equation*} (\ell+m)\mu(\Sigma \times E) + \ell \sum_{j=0}^{q} \mu(S^{-j}E \times \Sigma) \geq (\ell+m) \left((k+1)\cdot \overline{\diff}(A)-k \right)+\ell (k+1)\cdot \overline{\diff}(A) + \ell(q-k).
 \end{equation*}
It also holds that $A=\{n\in \N: S^n a \in E \}$ and $(A-j)/(\ell+m ) =\{n\in \N: S^{(q+1)\ell n+j}a' \in E\}$, for each $j=0,1,\dots,\ell+m -1$, where $(A-j)/(\ell+m )=\{n\in \N: n(\ell+m )+j\in A\}$.
\end{lemma}

\begin{lemma}\label{FCP no shift general k}
Let $A\subset \N$ and $\ell,m\in \N$ with $k=m/\ell$. Then, there exists an ergodic system $(\Sigma 
\times \Sigma, \mu, S \times S)$, an open set $E\subset \Sigma$, a pair of points $a'',a\in \Sigma$ and a F\o lner sequence 
$\Phi$, such that $(a'',a)\in \textbf{gen}(\mu,\Phi)$ and 
\begin{equation*} 
(\ell+m)\mu(\Sigma \times E) + \ell \mu(E \times \Sigma) \geq  (2\ell+m)\left( (k+1)\cdot \overline{\diff}(A) - k \right).
\end{equation*}
It also holds that $A = \{ n \in \N: S^n a \in E \}$ and $A/(\ell+m ) =\{n\in \N: S^{\ell n}a'' \in E\}$. 
\end{lemma}

We postpone the proofs of Lemmas \ref{FCP shift general k} and 
\ref{FCP no shift general k} until the end of Section \ref{setup}. Instead, 
we will finish this section by showing how to deduce Theorems \ref{mB+lB unrestricted shift} and \ref{mB+lB unrestricted no shift} from their dynamical counterparts, using the tools we have acquired thus far. To this end, we reverse the order of presentation and start with the case of no shift because the proof is, at the very least notationally, lighter.

\begin{proof}[Proof that \cref{dynamical mB+lB unrestricted no shift} implies \cref{mB+lB unrestricted no shift}.]
Let $\ell,m \in \N$ and let $k=m/\ell$. 
Given $A\subset \N$ with $\overline{\diff}(A)>1-1/\left(\ell(k+1)(k+2)\right)$, 
we find, by way of \cref{FCP no shift general k}, an ergodic system $(\Sigma 
\times \Sigma, \mu, S \times S)$, an open set $E\subset 
\Sigma$, a pair of points $a'',a\in \Sigma$ satisfying the conditions of \cref{FCP no shift general k} and a F\o lner sequence 
$\Phi$, such that $(a'',a)\in \textbf{gen}(\mu,\Phi)$ and 
\begin{equation*} 
(\ell+m)\mu(\Sigma \times E) + \ell \mu(E \times \Sigma) \geq  (2\ell+m) \left( (k+1)\cdot \overline{\diff}(A)-k \right)>(2\ell+m)(1-\frac{1}{\ell(k+2)})=2\ell+m-1,
\end{equation*}    
because $\ell+m=\ell(k+1)$, hence $\ell(k+2)=2\ell+m$. It 
follows by \cref{dynamical mB+lB unrestricted no shift} that 
there 
exist some points $(x_{10},x_{11}),(x_{20},x_{21}) \in \Sigma\times \Sigma$ 
so that $((a'',a),(x_{10},x_{11}),(x_{20},x_{21})) \in (\Sigma\times \Sigma)^3$ is an $(\ell,m)$-Erd\H{o}s progression for $(\Sigma 
\times \Sigma, \mu, S \times S)$ and $((x_{10},x_{11}),(x_{20},x_{21})) \in (E\times \Sigma) \times (\Sigma \times E)$. Then, an application of \cref{EP and return times} yields an infinite set $B\subset \N$ such that 
$$ B\subset \{n\in \N: (S \times S)^{\ell n}(a'',a) \in E\times \Sigma \}= \{n\in \N: S^{\ell n}a'' \in E \} $$
and similarly,
$$\{ mb_1+\ell b_2: b_1,b_2\in B\ \text{and}\ b_1<b_2\} \subset \{n\in \N: (S \times S)^{n}(a,a) \in \Sigma \times E \}=\{n\in \N: S^na \in E \}.$$
Since $S^{\ell n}a''\in E\iff (\ell+m) n\in A$, the former inclusion rewrites as 
$(\ell+m)B \subset A$ and the latter as $\{ mb_1+\ell b_2: b_1,b_2\in B\ \text{and}\ b_1<b_2\} \subset A$. Combining these two we conclude that $\{mb_1+\ell b_2: b_1,b_2\in B\ \text{and}\ b_1\leq b_2\} \subset A$.
\end{proof}

\begin{proof}[Proof that \cref{dynamical mB+lB unrestricted shift} implies \cref{mB+lB unrestricted shift}.]
Let $m,\ell \in \N$ and $k=m/\ell$, $q=\lceil k \rceil$.  
Given $A\subset \N$ with $\overline{\diff}(A)>(k+1)/(k+2)$, 
we find, by way of \cref{FCP shift general k}, an ergodic system $(\Sigma 
\times \Sigma, \mu, S^{(q+1)} \times S)$, an open set $E\subset 
\Sigma$, a pair of points $a',a\in \Sigma$ satisfying the conditions in \cref{FCP shift general k}  
and a F\o lner sequence 
$\Phi$, such that $('a,a)\in \textbf{gen}(\mu,\Phi)$ and
\begin{multline*}
(\ell+m)\mu(\Sigma \times E) + \ell \sum_{j=0}^{q} \mu(S^{-j}E \times \Sigma) \geq (2\ell+m) (k+1)\cdot \overline{\diff}(A) + \ell(q-k) -(\ell+m)k >\\
(2\ell+m)(k+1)\frac{k+1}{k+2}+\ell(q-k)-\ell(k+1)k=\ell(k+1)(k+1)+\ell(q-k)-\ell(k+1)k=\ell(q+1),   
\end{multline*}
We now justify why Theorem \ref{dynamical mB+lB unrestricted shift} applies in order for us to recover an $(\ell,m)$-Erd\H{o}s progression $\left( (a',a),(x_{10},x_{11}) , (x_{20},x_{21}) \right) \in (\Sigma \times \Sigma)^3$, with
$$(x_{10},x_{11}, x_{20},x_{21}) \in E_j \times F_j=S^{-(j-1)}E \times \Sigma \times \Sigma \times S^{-(j-1)}E,$$ 
for some $j\in \{1,2,\dots,\ell+m \}$. To see this, note that $E_j=S^{-(j-1)}E\times S$ and so $E_{j+(q+1)}=(S^{(q+1)}\times S)^{-1}E_j$, for $j\in \{1,\ldots,(\ell-1)(q+1)\}$. Moreover we showed above that
$$ (\ell+m)\mu(\Sigma \times E) + \ell \sum_{j=0}^{q} \mu(S^{-j}E \times \Sigma) > \ell(q+1),$$
which is precisely \eqref{mB+lB shift eq}.
Then, using \cref{EP and return times} we find an infinite set $B\subset \N$ such that
$$ B\subset \{n\in \N: S^{(q+1)\ell n}a' \in S^{-(j-1)}E \} $$
and 
$$\{ mb_1+\ell b_2: b_1,b_2\in B\ \text{and}\ b_1<b_2\} \subset \{n\in \N: S^na \in S^{-(j-1)}E \},$$
for some $j\in \{1,\ldots,\ell+m\}$. 
From the defining properties of $a'$, we see that the former becomes $B \subset (A-(j-1))/(\ell+m ).$ Thus, unraveling the definitions, we see that these two inclusions 
together translate to 
$$\{mb_1+\ell b_2: b_1,b_2\in B\ \text{and}\ b_1\leq b_2\} \subset A-(j-1),$$
and so we conclude. 
\end{proof}

\section{The ergodic theory setup} \label{setup}
\subsection{An overview}

A sufficient condition for a triple 
$(x_0,x_1,x_2) \in X^3$ to be an 
$(\ell,m)$-Erd\H{o}s progression is that 
$(x_0,x_1)$ is a $(T^{\ell} \times T^m)$-generic 
point for some invariant measure and 
$(x_1,x_2) \in X\times X$ is in the support of 
that measure. This is a general fact which can 
easily be deduced from the definitions. A well-known
consequence of the mean ergodic theorem is that for an
ergodic system $(Y,\nu,S)$ and any F\o lner 
sequence, there is a subsequence $\Phi$ such that 
$\nu$-almost every point $y\in Y$ is $S$-generic 
along $\Phi$. Hence, if in the above setting 
$\mu$ is a $T$-invariant measure, we want to 
consider an ergodic decomposition of $\mu \times \mu$. 

As in \cite{kmrr2} we are interested in 
progressions with prescribed first coordinate 
$a\in X$ and 
so we will reduce to the case that the ergodic 
decomposition is continuous. However, our work is 
different here because we need typical points to 
be $(T^{\ell} \times T^m)$-generic for general 
$\ell,m \in \N$. Another important aspect of this 
problem is that the linear patterns we are 
looking for are still dynamically controlled by 
the Kronecker factor. Therefore, we find it useful to 
introduce a measure 
$\sigma$ on $X\times X$ which gives full measure 
to the set of points $(x_1,x_2)$ such that 
$(a,x_1,x_2)$ projects to an $(\ell,m)$-three 
term progression on the Kronecker, in a way 
similar to that done in \cite{kmrr2} for the case $\ell=m=1$. 

\subsection{Continuous ergodic decomposition}\label{cts ergodic decomposition section}

To proceed with the constructions we briefly 
recall some standard notions. 
If $(X,\mu,T)$ and $(Y,\nu,S)$ are two systems, a 
measurable map $\pi\colon X \to Y$ for which $\pi \mu = \nu$ and \footnote{$\pi \mu$ denotes the pushforward of $\mu$ by $\pi$} 
\begin{equation} \label{conjugation}
\pi \circ T = S \circ \pi  \quad \mu \text{-almost everywhere} 
\end{equation}
is called a factor map. If, in addition, $\pi$ is 
continuous, surjective and \eqref{conjugation} holds everywhere we 
call $\pi$ a continuous factor map. Note that factors of ergodic systems are also ergodic. 

A group rotation is a system $(Z,\nu,R)$, for a compact abelian 
group $Z$ with its normalized Haar measure $\nu$ and $R\colon Z \to Z$ 
being a rotation of the form $R(z)=z+b$, some $b \in Z$. In this case we can also assume that the compatible metric on $Z$ is such that $z \mapsto z+w$ is an isometry for all $w\in Z$.

Every ergodic system has a maximal group rotation 
factor, called the Kronecker factor, and while in 
general the factor map from an ergodic system 
$(X,\mu,T)$ to its Kronecker $(Z,\nu,R)$ is only 
measurable, for our purposes we may assume that 
it is also a continuous surjection. Indeed, using 
Proposition $3.20$ from \cite{kmrr1} one can show 
that Theorem \ref{dynamical mB+lB} follows from 
the next seemingly weaker result. The proof of 
this implication is the same as the proof that 
Theorem $3.2$ implies Theorem $1.4$ in 
\cite{kmrr2} or the proof of Theorem $2.1$ via 
Theorem $3.4$ in \cite{kousek_radic2024}.

\begin{theorem} \label{continuous dynamical mB+lB}
Let $(X,\mu,T)$ be an ergodic system and assume 
there is a continuous factor map $\pi$ to its 
Kronecker. Let $a\in \textbf{gen}(\mu,T,\Phi)$, 
for some F\o lner sequence $\Phi$ and 
$E \subset X$ be an open set with $\mu(E)>0$. 
Then, there exist $x_1,x_2\in X$ and $t\in \N$ so that  $(a,x_1,x_2)\in X^3$ is an $(\ell, m)$-Erd\H{o}s 
progression such that $T^tx_2 \in E$.  
\end{theorem}

In a similar fashion, Theorems \ref{dynamical mB+lB unrestricted shift} and \ref{dynamical mB+lB unrestricted no shift} follow 
from the next seemingly weaker results, respectively, where the 
system is assumed to have a continuous Kronecker factor map (essentially, the proof of Theorem $2.1$ from Theorem $3.4$ in \cite{kousek_radic2024} contains one of the analogous arguments in the case $m=\ell=1$).

\begin{theorem}\label{continuous dynamical mB+lB unrestricted shift}
Let $(X,\mu,T)$ be an ergodic system and assume there is a continuous factor map $\pi$ to its Kronecker and $a\in \textbf{gen}(\mu,T,\Phi)$ for some F\o lner 
sequence $\Phi$. Moreover, let $\ell,m\in \N$, $q=\lceil m/\ell \rceil$ and assume that 
$E_1,\ldots,E_{\ell+m},F_1,\ldots,F_{\ell+m}\subset X$ are open sets 
such that $F_j=T^{-(j-1)}F_1$, $j=1,\ldots,\ell+m$, 
$E_{i+(q+1)}=T^{-1}E_i$, for $i=1,\ldots,\ell+m-q-1$, and also  
\begin{equation} \label{cts mB+lB shift eq}
(\ell+m)\mu(F_1)+\ell \left(\mu(E_1)+\cdots + \mu(E_{q+1}) \right)>\ell(q+1).     
\end{equation}
Then, for some $j\in \{1,\ldots,\ell+m\}$, there exist $x_1,x_2 \in X$ so that $(a,x_1,x_2) \in X^3$ is an $(\ell,m)$-Erd\H{o}s progression 
and $(x_1,x_2) \in E_j \times F_j$. 
\end{theorem}

\begin{theorem}\label{continuous dynamical mB+lB unrestricted no shift}
Let $(X,\mu,T)$ be an ergodic system and assume 
there is a continuous factor map $\pi$ to its 
Kronecker. Let $a\in \textbf{gen}(\mu,T,\Phi)$ for some F\o lner 
sequence $\Phi$ and $E,F \subset X$ be an open sets with
\begin{equation}\label{cts mB+lB no shift eq}
(\ell+m)\mu(F)+\ell\mu(E)>2\ell+m-1.    
\end{equation} 
Then, there exist $x_1,x_2 \in X$ and $t\in \N$ so that $(a,x_1,x_2) \in X^3$ is an $(\ell,m)$-Erd\H{o}s progression and $(x_1,x_2) \in E \times F$.  
\end{theorem}

From now on, we fix an ergodic system $(X,\mu,T)$
and assume $\pi$ is a continuous factor map to
its Kronecker, $(Z,\nu,R)$. We also fix a 
disintegration $z \mapsto \eta_z$ of $\mu$ over 
the Kronecker (for details, see for example, 
\cite[Theorem $5.14$]{Einsiedler&Ward:2011}). Then, for every 
$(x_1,x_2)\in X \times X $ we define the measure
\begin{equation}\label{lambda}
\lambda_{(x_1,x_2)}=\int_Z \eta_{\ell z+\pi(x_1)} 
\times \eta_{mz + \pi(x_2)}\ d\nu(z)    
\end{equation}
on $X\times X$. We stress that \eqref{lambda} is well-defined since, for each $(x_1,x_2) \in X\times X$ the measures $\eta_{\ell z + \pi(x_1)}$ and $\eta_{mz+\pi(x_2)}$ are defined for $\nu$-almost every $z\in Z$. The last claim holds because ergodicity of $R$ implies that the subgroups $\ell Z$ and $mZ$ of $Z$ both have positive measure (see the proof of Lemma \ref{projections of sigma} for more details on this).
We next examine some properties 
of this (a posteriori) disintegration of 
$\mu \times \mu$.

\begin{proposition} \label{erg decomposition}
In the above setting, the map $(x_1,x_2) \mapsto 
\lambda_{(x_1,x_2)}$ satisfies the following 
properties. 
\begin{enumerate}[(i).]
     \item The map $(x_1,x_2) \mapsto \lambda_{(x_1,x_2)}$ is continuous. 
    \item The map $(x_1,x_2) \mapsto \lambda_{(x_1,x_2)}$ is a disintegration of $\mu \times \mu$, meaning that 
    $$\int_{X\times X} \lambda_{(x_1,x_2)}\ d(\mu \times \mu)(x_1,x_2) = \mu \times \mu. $$
    \item For $(\mu \times \mu)$-almost every $(x_1,x_2)\in X\times X$, the point $(x_1,x_2)$ is $(T^{\ell} \times T^m)$-generic for $\lambda_{(x_1,x_2)}$ and $\lambda_{(x_1,x_2)}$ is $(T^{\ell} \times T^m)$-ergodic.
    \item For every $(x_1,x_2) \in X\times X$, we have that $\lambda_{(x_1,x_2)}=\lambda_{(T^{\ell}x_1,T^mx_2)}.$ 
 \end{enumerate}  
\end{proposition}
The rest of this subsection is 
devoted to the proof of Proposition 
\ref{erg decomposition}. The first step is a 
result showing that in some sense the Kronecker 
is a characteristic factor.

\begin{proposition} \label{kronecker is characteristic}
Fix an ergodic system $\xmt$ with Kronecker factor $(Z,\nu,R)$ and factor map $\pi: X \to Z$. Then, for any $\ell,m \in \N$ and $f,g\in L^{\infty}(X,\mu)$ we have that 
\begin{equation} \label{kronecker equation}
\lim_{N\to \infty} \frac{1}{N} \sum_{n=1}^N f(T^{\ell n}x_0)\cdot g(T^{m n}x_1) = \lim_{N\to \infty} \frac{1}{N} \sum_{n=1}^N \mathbb{E}(f | Z)(R^{\ell n}\pi(x_0)) \cdot \mathbb{E}(g | Z)((R^{m n}\pi(x_1)),
\end{equation}
for $(\mu \times \mu)$-almost every $(x_0,x_1)\in X\times X$.    
\end{proposition}

\begin{proof}
Both limits in \eqref{kronecker equation} exist 
by the pointwise ergodic theorem, so we simply 
need to establish their equality in the $L^2$ 
norm. By the Jacobs-de Leeuw-Glicksberg 
decomposition (see \cite[Theorem $2.24$]{DKE}) 
this reduces to showing that whenever either $f$ 
or $g$ is a weak mixing function, 
then 
\begin{equation} \label{weak mixing vanishes}
\lim_{N\to \infty} \frac{1}{N}\sum_{n=1}^N T^{\ell n}f \otimes T^{m n}g =0,
\end{equation}
in $L^2(\mu \times \mu)$. Assuming, without loss 
of generality, that $f$ is the weak mixing
function and 
setting $u_n=T^{\ell n}f \otimes T^{m n}g$ this 
follows directly by the van der Corput lemma (originally proven in this version in \cite[Theorem $1.4$]{Berg.2}), for 
$$\left| \frac{1}{K}\sum_{k=1}^K \frac{1}
{N}\sum_{n=1}^N \langle u_{n+k} , u_n \rangle 
\right| = \left| \frac{1}{K}\sum_{k=1}^K \int_X 
T^{\ell k}f \cdot \overline{f}\ d\mu 
\int_X 
T^{m k}g \cdot \overline{g}\ d\mu  \right| \leq 
\norm{g}^2_{\infty} \frac{1}{K}\sum_{k=1}^K 
\left| \int_X T^{\ell k}f \cdot \overline{f}\ 
d\mu \right|,$$
which goes to $0$ as $K\to \infty$ by the 
definition of weak mixing functions. 
\end{proof}

We are now in the position to prove the main result of this subsection.
\begin{proof}[Proof of \cref{erg decomposition}]
We begin by showing the map 
$(x_1,x_2) \mapsto \lambda_{(x_1,x_2)}$ is a disintegration 
of $\mu \times \mu$. Let $f,g\in L^{\infty}(\mu)$. Then,  
\begin{align*}
& \int_{X^2} f\otimes g\ d\lambda_{(x_1,x_2)} d(\mu\times \mu)(x_1,x_2) 
= \int_{X^2} \int_{Z} \int_{X} f\ d\eta_{\ell z+\pi(x_1)} \int_X g\ d\eta_{mz+\pi(x_2)}\ d\nu(z) d\mu(x_1) d\mu(x_2)  \\
& = \int_Z \left( \int_X \int_{X} f\ d\eta_{\ell z+\pi(x_1)}\  d\mu(x_1) \times \int_X \int_{X} f\ d\eta_{m z+\pi(x_2)}\ d\mu(x_2) \right)\ d\nu(z) \\
& = \int_Z \left( \int_{X} f\ d\mu \times \int_{X} g\  d\mu\ \right)\ d\nu(z) = \int f\otimes g\ d(\mu\times \mu),
\end{align*}
because of \eqref{lambda} and the fact that for each $w\in Z$,
$$ \mu =\int_Z  \eta_z\ d\nu(z) = \int_Z \eta_z\ d(\pi \mu)(z) = \int_X \eta_{\pi(x)}\ d\mu(x)  = \int_X \eta_{w+\pi(x)}\  d\mu(x).$$
Part $(ii)$ follows by standard approximation arguments using Stone-Weierstrass' theorem and the Riesz-Markov-Kakutani representation theorem.  

To prove that $(x_1,x_2) \mapsto \lambda_{(x_1,x_2)}$ is 
continuous we need to show that for each $F\in C(X\times X)$, the map $(x_1,x_2) \mapsto \int_{X\times X} F\ d\lambda_{(x_1,x_2)}$ is continuous, because -- implicitly -- the topology we endow the space of Borel measures on $X\times X$ with is 
the weak* topology. By another application of 
Stone-Weierstrass' theorem, we can assume that 
$F\in C(X\times X)$ in the previous is of the form 
$f\otimes g$, for some $f,g\in C(X)$. To this end, we first let $f,g\in C(Z)$, and then these functions are also uniformly continuous (by compactness) and so the map 
$$(v,w) \mapsto \int_{Z}f(\ell z+v)\ g(mz+w)\ d\nu(z) $$
is continuous. Thus, the density of $C(Z)$ in $L^2(Z,m)$ 
implies the continuity of the analogous map for 
$f,g \in L^{2}(Z,m)$. 

Now, if $f,g\in C(X)$, we have that $\mathbb{E}[f|Z], \mathbb{E}[g|Z] \in L^2(Z,m)$ and so we see that 
$$(x_1,x_2) \mapsto \int_{Z} \mathbb{E}[f|Z](\ell z+\pi(x_1))\ \mathbb{E}[g|Z](mz+\pi(x_2))\ d\nu(z) $$
is continuous as the composition of continuous maps. Noting 
that $\mathbb{E}[f|Z](z)=\int_{X} f\ d\eta_z$ for $\nu$-almost every $z\in Z$, we see that
$$\int_{X\times X} f\otimes g\ d\lambda_{(x_1,x_2)}= \int_{Z} \mathbb{E}[f|Z](\ell z+\pi(x_1))\ \mathbb{E}[g|Z](mz+\pi(x_2))\ d\nu(z)$$
and so the continuity of $(x_1,x_2) \mapsto \lambda_{(x_1,x_2)}$ follows.

Property $(iv)$ is immediate up to null sets because 
$\pi \circ T = R\circ \pi$ as $\pi$ is a factor map and 
then, the established continuity of the decomposition 
implies it for all points $(x_1,x_2) \in X\times X$.

We are only left with proving that 
$(x_1,x_2) \mapsto \lambda_{(x_1,x_2)}$ is a $(T^\ell\times T^m)-$ergodic decomposition of $\mu \times \mu$, because then part $(iii)$ follows as a consequence of the pointwise ergodic theorem (see, for example, \cite[Corollary $2.9$]{kmrr1}). In other words, we have to show that for each bounded and measurable $F: X\times X \to \C$ it holds that 
$$ \int_{X\times X} F\ d\lambda_{(x_1,x_2)} = \mathbb{E}[F | \I](x_1,x_2),$$
for $(\mu\times \mu)-$almost every $(x_1,x_2)\in X\times X$, where $\I$ denotes the $\sigma$-algebra of $(T^{\ell}\times T^m)-$invariant sets on $X\times X$. By the ergodic theorem, this is equivalent to showing that
$$\lim_{N\to \infty} \frac{1}{N}\sum_{n=1}^N F(T^{\ell n}x_1,T^{mn}x_2)=\int_{X\times X} F\ d\lambda_{(x_1,x_2)},$$
for $(\mu\times \mu)-$almost every $(x_1,x_2)\in X\times X$ and by standard approximation arguments
this reduces to showing that
\begin{equation*}
\lim_{N\to \infty} \frac{1}{N}\sum_{n=1}^N f(T^{\ell n}x_1)\cdot g(T^{mn}x_2)=\int_{X\times X} f\otimes g\ d\lambda_{(x_1,x_2)},
\end{equation*}
for $(\mu\times \mu)-$almost every $(x_1,x_2)\in X\times X$ and every $f,g\in L^{\infty}(X,\mu)$. By \eqref{kronecker equation}, it suffices to show that  
\begin{equation} \label{eq: 1}
 \lim_{N\to \infty} \frac{1}{N} \sum_{n=1}^N \mathbb{E}(f | Z)(R^{\ell n}\pi(x_1)) \cdot \mathbb{E}(g | Z)((R^{m n}\pi(x_2))=\int_{X\times X}\ f\otimes g\ d\lambda_{(x_1,x_2)},
\end{equation}
for $(\mu \times \mu)$-almost every $(x_1,x_2)\in X\times X$. 

With this reduction, the algebraic structure of rotations in compact abelian 
groups (see, for example, \cite[Chapter $4$]{Host-Kra}) 
allows us to conclude. More precisely, $(R^n(0))_{n\in \Z}$ 
is equidistributed in the compact abelian group $Z$ and the 
function $\phi: Z \to \C$ defined by
$$\phi(z)= \mathbb{E}(f | Z)(\ell z+\pi(x_1)) \cdot \mathbb{E}(g | Z)((mz+\pi(x_2))$$
is Riemann integrable, and therefore the limit on the left hand side of \eqref{eq: 1} becomes 
$$  \lim_{N\to \infty} \frac{1}{N} \sum_{n=1}^N \phi(R^n(0))= \int_Z \phi(z)\ d\nu(z)= \int_Z \mathbb{E}(f | Z)(\ell z+\pi(x_1)) \cdot \mathbb{E}(g | Z)((mz+\pi(x_2))\ d\nu(z).$$
But, as we saw before, unraveling the definition of $\lambda_{(x_1,x_2)}$ gives
$$ \int_Z \mathbb{E}(f | Z)(\ell z+\pi(x_1)) \cdot \mathbb{E}(g | Z)((mz+\pi(x_2))\ d\nu(z) = \int_{X\times X} f\otimes g\ d\lambda_{(x_1,x_2)}$$
and thus, \eqref{eq: 1} follows.
\end{proof}

\subsection{A measure on $(\ell,m)-$Erd\H{o}s progressions and some of its properties}

As above, $\xmt$ is an ergodic system and $(Z,\nu,R)$
is its Kronecker factor with (continuous) factor 
map $\pi: X \to Z$, and $R$ is a rotation by some $b\in Z$. Moreover, we let 
$a\in \textbf{gen}(\mu,T,\Phi)$, for some \Folner\ sequence $\Phi$. We 
consider the measure 
\begin{equation}\label{sigma}
\sigma_a= \int_Z \eta_{\ell z+ \pi(a)} \times \eta_{(\ell+m) z + \pi(a)}\ d\nu(z)
\end{equation}
on $X \times X $. The first useful property of 
$\sigma_a$ which relates to the 
disintegration 
$(x_1,x_2) \mapsto \lambda_{(x_1,x_2)}$ reads as 
follows.

\begin{lemma} \label{equality of lambdas}
For $\sigma_a$-almost every $(x_1,x_2) \in X \times X$ it holds that $\lambda_{(a,x_1)}=\lambda_{(x_1,x_2)}$.     
\end{lemma}
\begin{proof}
It is obvious from the definition of $\sigma_a$ 
that the set $P$ defined by
$$P=\{(x_1,x_2) \in X\times X \colon \pi(x_1)=\ell w+\pi(a)\ \text{and}\ \pi(x_2)=(\ell+m)w+ \pi(a),\ \text{for some}\ w\in Z\}$$ 
has $\sigma_a(P)=1.$ We fix $(x_1,x_2) \in P$ and let $w\in Z$ be such that $\pi(x_1)=\ell w + \pi(a)$ and $\pi(x_2)=(\ell + m)w + \pi(a)$. 
The proof will be complete once we show that 
$\lambda_{(x_1,x_2)}=\lambda_{(a,x_1)}$. To this 
end, since $\nu$ is a shift invariant measure, 
making the change of variables $z \mapsto z-w$ we 
get that
$$\lambda_{(x_1,x_2)}=\int_Z \eta_{\ell z+\ell w + \pi(a)} \times \eta_{mz + (\ell+m)w+\pi(a)}\ d\nu(z)=\int_Z \eta_{\ell z+\pi(a)} \times \eta_{mz + \ell w + \pi(a)}\ d\nu(z),$$
which equals $\lambda_{(a,x_1)}$ since $\pi(x_1)=\ell w + \pi(a)$.
\end{proof}
The latter result illuminates why this particular definition of $\sigma_a$ is useful. The measure was defined so that it essentially only witnesses pairs $(x_1,x_2)\in X\times X$ whose projection on the Kronecker give rise to $3$-term $(\ell,m)$-progressions in $Z$ through $(\pi(a), \pi(x_1),\pi(x_2))$. This is apparent from the equalities $\pi(x_1)-\pi(a)=\ell w$ and $\pi(x_2)-\pi(x_1)=m w$, some $w\in Z$, which hold for any $(x_1,x_2)\in P.$ 

Then, a potential route to finding $(\ell,m)$-\Erdos\ progressions of the form $(a,x_1,x_2) \in X^3$ is laid out. It is sufficient to find a pair $(x_1,x_2)\in P$ with $(x_1,x_2)\in \supp{\lambda_{(x_1,x_2)}}$, and such that $(a,x_1)$ is $(T^{\ell}\times T^m)$-generic for $\lambda_{(a,x_1)}$, which coincides with $\lambda_{(x_1,x_2)}$. To this end, we shall also need two relations between the 
measure $\mu$ and push-forwards of projections of 
$\sigma_a$.

\begin{lemma} \label{projections of sigma}
Let $\pi_i: X\times X \to X$ denote the projection $(x_1,x_2) \mapsto x_i$ onto the $i$-th coordinate, where $i\in \{1,2\}$. Then, if $\pi_i \sigma_a$ denotes the push-forward of $\sigma_a$ by $\pi_i$, we have that  
$\frac{1}{\ell}\left( \pi_1 \sigma_a + T\pi_1 \sigma_a +\dots +T^{\ell-1}\pi_1 \sigma_a \right) = \mu$ and $\frac{1}{\ell+m}\left( \pi_2 \sigma_a + T\pi_2 \sigma_a +\dots+ T^{\ell+m-1}\pi_2 \sigma_a \right) = \mu$.
\end{lemma}

\begin{proof}
We only prove the first claim as the second one 
follows similarly. Let $\ell Z$ denote the 
subgroup $\{\ell z = z+z+\cdots+z\ \text{with 
$\ell$ summands} \colon z\in Z \}$ and let $\xi$ 
denote its Haar measure. Ergodicity of $R$ means that $\{R^n 0: n\in \N\}$ 
is dense in $Z$, and thus 
$Z=(\ell Z) \cup (R (\ell Z)) \cup \dots \cup 
(R^{\ell-1}(\ell Z))$. Therefore, there exists 
$w\in Z$ and $j\in \{0,1,\dots,\ell-1\}$ such 
that $\pi(a)=R^j(\ell w)=\ell w + jb$, so that 
$$ \pi_1 \sigma_a= \int_Z \eta_{\ell z+\pi(a)}\ d\nu(z)= \int_Z \eta_{\ell (z+w)+jb}\ d\nu(z)=\int_Z \eta_{\ell z+jb}\ d\nu(z)=\int_{\ell Z+jb} \eta_{u}\ d(R^j \xi)(u).$$
Finally, $T^i \eta_u=\eta_{R^i u}$ and $R^{\ell}\xi = \xi$, which implies that
$$\frac{1}{\ell} T^{i}\pi_1 \sigma_a = \int_{\ell Z+jb} T^i \eta_{u}\ d\frac{1}{\ell}(R^j \xi)(u)=\int_{\ell Z+(j+i) b} \eta_{u}\ d\frac{1}{\ell}(R^{j+i} \xi)(u),$$
for each $i\in \{0,1,\ldots,\ell-1\}$, where $j+i$ is taken mod $\ell$ and thus 
$$\frac{1}{\ell}\left( \pi_1 \sigma_a + T\pi_1 \sigma_a +\dots +T^{\ell-1}\pi_1 \sigma_a \right)=  \int_{Z} \eta_{z}\ d\frac{1}{\ell}(\xi+R\xi+\dots + R^{\ell-1}\xi)(z)=\mu.$$
\end{proof}

\subsection{Support of the measure on $(\ell,m)-$Erd\H{o}s progressions}

We begin with some notational remarks. Recall our 
setting; We have an 
ergodic system $(X,\mu,T)$ with a continuous 
factor map $\pi$ to its Kronecker $(Z,\nu,R)$, say 
$R$ is the rotation by some $b\in Z$, and a 
generic point $a\in \textbf{gen}(\mu,T,\Phi)$ for 
some \Folner\ sequence $\Phi$. We write $z \to \eta_z$ for the disintegretion of $\mu$ over $\pi$.
Then if $k\in \N$ is any, we can define 
$X_i=\pi^{-1}(R^i(k Z))$, for $i=0,1,\dots,k-1$ 
and consider the ergodic components of $\mu$ for 
the transformation $T^{k}$ given by 
\begin{equation} \label{def of mu_i}
\mu_i=\int_Z \eta_{k z+ib}\ d\nu(z)=\int_{k Z+ib} \eta_u\ d(R^i \zeta)(u),
\end{equation}
where $\zeta$ is the Haar measure on $k Z$.
This subsection is devoted to showing that if
\begin{equation} \label{S}
S=\{(x_1,x_2) \in X \times X: (x_1,x_2) \in \supp(\lambda_{(x_1,x_2)})\},     
\end{equation}
then $\sigma_a(S)=1$. We begin with an extension 
of Proposition $3.10$ of \cite{kmrr2}.  

\begin{remark} \label{pi(x) notation}
In the above setting, for any $x\in X_i$ we have that 
$\pi(x)=k w+ib$, some $w\in Z$, $i=0,1,\ldots,k-1$.    
\end{remark}

For the next result we use notation from \cite{kmrr2}. In particular, $\mathcal{F}(X)$ 
denotes the family of non-empty and closed 
subsets of the compact metric space $(X,d)$, 
endowed with the Hausdorff metric, denoted by $\mathbf{H}$. 

\begin{proposition} \label{Lebesgue density theorem}
Fix a system $(X,\mu,T)$ and a continuous factor map $\pi$ to its 
Kronecker factor $(Z,\nu,R)$. Also fix $k\in \N$ and a disintegration $z \mapsto \eta_z$ over $\pi$. There is a sequence
$\delta(j) \to 0$, such that for $\mu$-almost every $x\in X$ (with 
$\pi(x)=k w+ib$ as in Remark \ref{pi(x) notation} ) the following holds: for every 
neighbourhood $U$ of $x$ we have 

\begin{equation} \label{eq quotient measures 2}
    \lim_{j\to \infty} \frac{m \left(\{z \in Z: \eta_{k z + ib} (U) > 0 \} \cap B(w, \delta(j))\right)}{m (B(w, \delta(j)))}=1.
\end{equation}   
\end{proposition}

\begin{proof}
We consider the maps $\Phi_{i}: Z \to \cF(X)$ 
defined by $\Phi_{i}(z)=\supp(\eta_{k z+ib})$, for any $z\in Z$ and $i \in \{0,1,2,\ldots,k-1\}$.
These 
maps are Borel measurable as the composition of 
three Borel measurable maps, $z \mapsto \ell z +ib$, $z\mapsto \eta_z$ and $\nu \mapsto \supp(\nu)$ (the latter is measurable in way of Lemma
$3.8$ of \cite{kmrr2}).

Just like in the proof of Proposition $3.10$ of \cite{kmrr2}, by Lusin's theorem, for each $j\in \N$ there is a closed set $Z_{i,j} \subset Z$ with $m(Z_{i,j})>1-2^{-j}$ and a $\delta(i,j)>0$ so that for all $z_1,z_2 \in Z_{i,j}$,
$$d(z_1,z_2)\leq \delta(i,j) \implies \mathbf{H}(\Phi_i(z_1) , \Phi_i(z_2))< \frac{1}{j}.$$
If we consider the sets
$$K_{i,j}=\{ z\in Z_{i,j} \colon m\left(B(z,\delta(i,j)) \cap Z_{i,j} \right) > \left( 1-\frac{1}{j} \right)m\left( B(z,\delta(i,j) \right) \}$$
and let 
$K_{i}=\bigcup_{M\geq 1} \bigcap_{j\geq M} K_{i,j}$ 
it follows by the same argument as in the proof 
of Proposition $3.10$ in \cite{kmrr2} that $m(K_{i})=1$. 

Next, we let $L'_{i}=\{x\in X: x\in \supp(\eta_{\pi(x)}) \} \cap \pi^{-1}(k K_{i}+ib)$. By the above we see that $m(k K_{i})=m(k Z)$ and since $\mu(\{x\in X: x\in \supp(\eta_{\pi(x)}) \})=1$ (see Lemma $3.9$ in \cite{kmrr2}) it follows that $\mu_i(L'_i)=1$. Thus, setting $L'=L'_0 \cup L'_1 \cup \dots \cup L'_{k-1}$ we have $\mu(L')=1$. 

Fix $x\in L'$, and $i\in \{0,1,\ldots,k-1\}$ so that 
$x\in L'_i$ and let $U$ be an open neighbourhood of $x$. In this case $\pi(x)=k w+ib$, some $w\in K_i$, because $\pi(x) \in k K_i+ib$. Now, as $w\in K_i$ and $U$ is open there is $j_0 \in \N$ such that $w\in K_{i,j}$ and $B(x,1/j) \subset U$, for all $j\geq j_0$. We claim that 
\begin{equation} \label{support inclusion}
B(w,\delta(i,j)) \cap Z_{i,j} \subset H:= \{z\in Z: \eta_{kz+ib}(U)>0\}.   \end{equation}
Indeed, let $w' \in B(w,\delta(i,j)) \cap Z_{i,j}$. As $d(w',w)<\delta(i,j)$ and $w',w$ are continuity points for $\Phi_i$, we see that $\mathbf{H}(\Phi_i(w),\Phi_i(w'))<1/j.$ Then, because $x\in \Phi_i(w)$, there exists $x'\in \Phi_i(w')$ with $d(x,x')<1/j$. This of course implies that $x' \in U$ and so $x'\in U \cap \Phi_i(w')$. As $\Phi_i(w')=\supp(\eta_{kw'+ib})$ it follows that $\eta_{kw'+ib}(U)>0$, that is $w' \in H$. 
As $w\in K_{i,j}$ it follows from \eqref{support inclusion} and by the construction of $Z_{i,j}$ that
$$\frac{m(H\cap B(w,\delta(i,j))}{m(B(w,\delta(i,j))} \geq 1-\frac{1}{j},$$
for all $j\geq j_0$. This then implies that
$$\lim_{j\to \infty} \frac{m(H\cap B(w,\delta(i,j))}{m(B(w,\delta(i,j))} = 1.$$
Setting $\delta(j)=\min_{\{i=0,1,\ldots,k-1\}} \delta(i,j)$, 
for each $j\in \N$, gives \eqref{eq quotient measures 2}.
\end{proof}

We are now in the position to prove the main 
result of this section. 

\begin{proposition} \label{support of sigma}
Let $S=\{(x_1,x_2) \in X \times X: (x_1,x_2) \in 
\supp (\lambda_{(x_1,x_2)})\}$ as in \eqref{S}. 
Then $\sigma_a(S)=1.$     
\end{proposition}

\begin{proof}
Using Proposition \ref{Lebesgue density theorem} above for $k=\ell$ and $k= m$, we find a 
sequence $\delta(j) \to 0$ and two sets $L,L' \subset X$ such that 
each point $x'\in L'$ satisfies \eqref{eq quotient measures 2} with $k=\ell$ and each $x\in L$ satisfies \eqref{eq quotient measures 2} with 
$k=m$. Now, $\mu(L)=\mu(L')=1$ and using Lemma \ref{projections of sigma} we see that $\sigma_a(L' \times L)=1$. To see this simply note that $L' \times L=(L' \times X) \cap (X \times L)$ and $\pi_1\sigma_a(L')=\pi_2\sigma_a(L)=1$.
We have thus reduced matters to showing that $L' \times L \subset S$. 

To this end, let $(x_0,x_1) \in L' \times L$ and $U_0, U_1$ be neighbourhoods of $x_0, x_1$ respectively. It suffices to verify that $\lambda_{(x_0,x_1)}(U)>0,$ where $U=U_0 \times U_1$. Writing $\pi(x_0)=\ell w_0+i_0b$ and $\pi(x_1)=mw_1+i_1b$ as in Remark \ref{pi(x) notation}, we have that
 \begin{align*}
        \lambda_{(x_0,x_1)} &= \int_Z \eta_{\ell z+\ell w_0+i_0b} \times \eta_{mz + mw_1+i_1b }  \diff \nu(z) 
    \end{align*}
and making the change of variables 
$z \mapsto z-w_0$ we see that 
\begin{align*}
        \lambda_{(x_0,x_1)}(U) &= \int_Z \eta_{\ell z+i_0b}(U_0) \times \eta_{mz + m(w_1-w_0)+i_1b }(U_1) \diff \nu(z). 
\end{align*}
Now, as $x_0 \in L'$ and $x_1 \in L$, there is $\delta>0$ such that
    \begin{equation} \label{eq half of the ball 0} 
        \frac{m \left(\{z \in Z: \eta_{\ell z+i_0b} (U_0) > 0 \} \cap B(w_0, \delta)\right)}{m (B(w_0, \delta))}\geq \frac{5}{6}
    \end{equation}
    and also 
    \begin{equation*} 
        \frac{m \left(\{z \in Z: \eta_{mz+i_1b} (U_1) > 0 \} \cap B(w_1, \delta)\right)}{m (B(w_1, \delta))}\geq \frac{5}{6}.
    \end{equation*}
    But $\{z \in Z: \eta_{mz+i_1b} (U_1) > 0 \} - (w_1-w_0) = \{z \in Z: \eta_{mz+m(w_1-w_0)+i_1b} (U_1) > 0 \}$ and so we get 
 \begin{equation} \label{eq half of the ball 1}
        \frac{m \left(\{z \in Z: \eta_{mz+m(w_1-w_0)+i_1b} (U_1) > 0 \} \cap B(w_0, \delta)\right)}{m (B(w_0, \delta))}\geq \frac{5}{6}.
    \end{equation}
    Finally, we consider the set $G$ defined by
    $$W=\{z\in Z: \eta_{\ell z+i_0b}(U_0)>0\ \text{and}\ \eta_{mz+m(w_1-w_0)+i_1b}(U_1)>0 \}.$$
    It is clear from \eqref{eq half of the ball 0} and \eqref{eq half of the ball 1} that $W$ contains at least half of the ball $B(w_0,\delta)$ and thus $m(W)>0.$ As for all $z \in W$ we have
    \begin{equation*}
        \eta_{\ell z+ i_0b} \times \eta_{mz + m(w_1-w_0)+ i_1b} (U_0 \times U_1) >0
    \end{equation*}
    and $m(W)>0$, it follows that $\lambda_{(x_0,x_1)}(U_0\times U_1) >0$, as desired.
    \end{proof}

\subsection{Proofs of correspondence principles} \label{FCP section for k in Q}

We move on to prove the correspondence principle-type of 
results stated in Lemmas \ref{FCP shift general k} and \ref{FCP no shift general k}. 
We start with the latter as it has a slightly simpler proof. 

Recall that $\Sigma$ denotes the space $\{0,1\}^{\Z}$ and is endowed with
the product topology so that it is compact metrizable. 
We also let $S \colon \Sigma \to \Sigma$ denote the shift transformation given by $S(x(n))=x(n+1)$, for any $n\in \Z$, $x=(x(n))_{n\in \Z} \in \Sigma$. We also recall the statements for convenience.

\begin{lemma**}
Let $A\subset \N$ and $\ell,m\in \N$ with $k=m/\ell$. Then, there exists an ergodic system $(\Sigma 
\times \Sigma, \mu, S \times S)$, an open set $E\subset \Sigma$, a pair of points $a'',a\in \Sigma$ and a F\o lner sequence 
$\Phi$, such that $(a'',a)\in \textbf{gen}(\mu,\Phi)$ and 
\begin{equation*} 
(\ell+m)\mu(\Sigma \times E) + \ell \mu(E \times \Sigma) \geq  (2\ell+m)\left( (k+1)\cdot \overline{\diff}(A) - k \right).
\end{equation*}
It also holds that $A = \{ n \in \N: S^n a \in E \}$ and $A/(\ell+m ) =\{n\in \N: S^{\ell n}a'' \in E\}$. 
\end{lemma**}

\begin{proof}
Let $A \subset \N$. By definition, there exists a sequence  $(N_i)$ of positive integers such that 
$$\overline{\diff}(A)=\lim_{i\to \infty} \frac{\left|A \cap [1,N_i]\right|}{N_i}.$$ 
We let $a\in \Sigma = \{0,1\}^{\Z}$ be the indicator of $A$, that is,
\begin{equation*}
    a(n) = 
    \left\{ 
    \begin{array}{cc}
        1 & \text{if $ n\in A$ } \\
        0 & \text{otherwise}.
    \end{array}
    \right.
\end{equation*}
Moreover, we let $a''\in \Sigma$ be defined by 
\begin{equation*}
    a''(n) = 
    \left\{ 
    \begin{array}{cc}
        a((\ell+m )i) & \text{if $ n=\ell i$,\ some $i\in \N$ } \\
        1 & \text{otherwise}.
    \end{array}
    \right.
\end{equation*}
If $E=\{x\in \Sigma\colon x(0)=1\}$, we observe 
that $E$ is clopen in $\Sigma$, $A=\{n\in \N\colon S^na \in E\}$ and also
$$\{n\in \N: S^{\ell n}a'' \in E\} = \{n\in \N: S^{(\ell+m )n}a\}= A/(\ell+m ).$$
Now let $N'_i=\lfloor N_i/(k+1)\rfloor$, for all $i\in \N$, $i\geq k+1$, and consider the sequence of Borel probability measures $(\mu_i)$ on $\Sigma \times \Sigma$ given by $$\mu_i=\frac{1}{N'_i} \sum_{n=1}^{N'_i} \delta_{(S \times S)^n(a'',a)}.$$
Letting $\mu'$ be a weak* accumulation point of $(\mu_i)$ 
we obtain an $(S^{(q+1)} \times S)$-invariant measure. 
It follows by definition that for all $i\in \N$, $i\geq k+1$, we have
\begin{equation} \label{ineq: 11}  
\mu_i(\Sigma \times E)=\frac{1}{N'_i} \sum_{n=1}^{N'_i} \delta_{S^na}(E) = \frac{\left| A\cap [1, N'_i ] \right| }{N'_i}.
\end{equation}
For any such $i$ we also have that
\begin{align}
    \mu_i(\Sigma \times E)&=\frac{1}{N'_i} \sum_{n=1}^{N'_i} \delta_{S^na}(E) = \frac{1}{N'_i} \left( \sum_{n=1}^{N_i} \delta_{S^na}(E) - \sum_{n=N'_i+1}^{N_i} \delta_{S^na}(E)  \right) \nonumber \\ 
& \geq \frac{k+1}{N_i} \left(  \sum_{n=1}^{N_i} \delta_{S^na}(E)- \frac{k}{k+1}N_i+o_{N_i\to \infty}(N_i) \right)= (k+1) \frac{\left|A \cap  [ 1, N_i ] \right|}{N_i}-k+o_{N_i\to \infty}(1) \label{ineq: 8}
\end{align} 
Taking limits in \eqref{ineq: 8} as $i\to \infty$ we have, by the definition of $\mu'$, the fact that $E\subset X$ is clopen and the choice of $(N_i)$ that
\begin{equation} \label{ineq: 12} 
\mu'(\Sigma \times E) \geq (k+1) \cdot \overline{\diff}(A)-k,
\end{equation}
On the other hand, for any $i\in \N$, $i\geq k+1$, we have that
\begin{align} 
\mu_i(E \times \Sigma )&=\frac{1}{N'_i} \sum_{n=1}^{N'_i} \delta_{S^{n}a''}(E) = \frac{1}{N'_i} \left( \sum_{n=1,n\notin \ell \N}^{N_i'} \delta_{S^{n}a''}(E) +  \sum_{n=1}^{\lfloor N'_i/\ell \rfloor }\delta_{S^{\ell n}a''}(E) \right)  \nonumber \\
&= \frac{1}{N_i'} \left( N_i'-\frac{N_i'}{\ell}+ \sum_{n=1}^{\lfloor N'_i/\ell \rfloor }\delta_{S^{(\ell+m ) n}a}(E) \right) \geq 1-\frac{1}{\ell}+ \frac{1}{N_i'} \sum_{n=1}^{\lfloor N_i/(\ell+m) \rfloor -2}\delta_{S^{(\ell+m ) n}a}(E)  \nonumber \\
& = 1-\frac{1}{\ell} + (k+1)\frac{|A\cap (\ell+m )\N \cap [1,N_i]|}{N_i} + o_{N_i\to \infty}(1) \label{ineq: 13},
\end{align}
where in the inequality we used that $\ell+m=\ell(k+1)$. Now, observe that 
$$|A\cap (\ell+m )\N \cap [1,N_i]| = |A\cap [1,N_i]| - |A\cap \left( \N\setminus (\ell+m )\N \right) \cap [1,N_i]| \geq |A\cap [1,N_i]|-N_i+\frac{N_i}{(\ell+m )},$$
so that
$$\liminf_{i\to \infty} \frac{|A\cap (\ell+m )\N \cap [1,N_i]|}{N_i} \geq \overline{\diff}(A)-1+\frac{1}{(\ell+m )}.$$
Using this and taking limits as $i\to \infty$ in \eqref{ineq: 13} we see that 
\begin{equation} \label{ineq: 14}
\mu'(E \times \Sigma ) \geq 1-\frac{1}{\ell}+(k+1)\left(\overline{\diff}(A)-1+\frac{1}{\ell+m }\right)=(k+1)\cdot \overline{\diff}(A)-k.
\end{equation}
Combining \eqref{ineq: 12} and \eqref{ineq: 14} we have that 
$$(\ell+m)\mu'(\Sigma \times E) + \ell \mu'(E \times \Sigma ) \geq (2\ell+m)\left( (k+1)\cdot \overline{\diff}(A) - k \right).$$
Although $\mu'$ is not necessarily ergodic, we 
can use its ergodic 
decomposition to find an $(S^{(q+1)} \times S)$-ergodic 
component of it, call it $\mu$, such that  
$$(\ell+m)\mu(\Sigma \times E) + \ell \mu(E \times \Sigma ) \geq (2\ell+m)\left( (k+1)\cdot \overline{\diff}(A) - k \right).$$
Without loss of generality we may 
assume that 
$\mu$ is supported on the orbit closure of $(a,a)$, 
since this holds for $\mu'$ by construction. Then by a standard argument (see \cite[ Proposition 3.9]{Fur1}) we see there is a F\o lner sequence $\Phi$ in $\N$, such that $(a,a) \in \textbf{gen}(\mu,\Phi)$. This completes the proof. 
\end{proof}

\begin{lemma*}
 Let $A\subset \N$ and $\ell,m \in \N$ with $k=m/\ell$ and let $q=\lceil k \rceil$, the ceiling of $k$. 
 Then, there exist an ergodic system $(\Sigma 
 \times \Sigma, \mu, S^{(q+1)} \times S)$, an open set $E\subset \Sigma$, a pair of points $a,a'\in \Sigma$ and a F\o lner sequence 
 $\Phi$, such that $(a',a)\in \textbf{gen}(\mu,\Phi)$ and 
 \begin{equation*} (\ell+m)\mu(\Sigma \times E) + \ell \sum_{j=0}^{q} \mu(S^{-j}E \times \Sigma) \geq (\ell+m) \left((k+1)\cdot \overline{\diff}(A)-k \right)+\ell (k+1)\cdot \overline{\diff}(A)+\ell(q-k).
 \end{equation*}
It also holds that $A=\{n\in \N: S^n a \in E \}$ and $(A-j)/(\ell+m ) =\{n\in \N: S^{(q+1)\ell n+j}a' \in E\}$, for each $j=0,1,\dots,\ell+m -1$, where $(A-j)/(\ell+m )=\{n\in \N: n(\ell+m )+j\in A\}$.
\end{lemma*}

\begin{proof}
Let $A \subset \N$. As before, let $(N_i)$ be a sequence of integers such that 
$$\overline{\diff}(A)=\lim_{i\to \infty} \frac{\left|A \cap [1,N_i]\right|}{N_i}$$ 
and $a\in \Sigma = \{0,1\}^{\Z}$ be the indicator of $A$. This time we let $a'\in \Sigma$ be defined by 
\begin{equation*}
    a'(n) = 
    \left\{ 
    \begin{array}{cc}
        a(\ell(k+1)i+j) & \text{if $ n=\ell(q+1)i+j$,\ some $i\in \N$, $j\in \{0,1,\dots,\ell(k+1)-1\}$ } \\
        1 & \text{otherwise}.
    \end{array}
    \right.
\end{equation*}
Observe that $a'$ is well-defined because $\ell(k+1) \leq \ell(q+1)$.
If $E=\{x\in \Sigma\colon x(0)=1\}$, we observe 
that $E$ is clopen in $\Sigma$, $A=\{n\in \N\colon S^na \in E\}$ and, since $(\ell+m )n=\ell(k+1)n$,
$$\{n\in \N: S^{(q+1)\ell n+j}a' \in E\} = \{n\in \N: S^{(\ell+m )n+j}a\}= (A-j)/(\ell+m ),$$ for each $j\in \{0,1,\dots,\ell+m -1\}$.
Now let $N'_i=\lfloor N_i/(k+1)\rfloor$, for all $i\in \N$ and consider the sequence of Borel probability measures $(\mu_i)$ on $\Sigma \times \Sigma$ given by $$\mu_i=\frac{1}{N'_i} \sum_{n=1}^{N'_i} \delta_{(S^{(q+1)} \times S)^n(a',a)},$$
if $i\geq k+1$, 
and let $\mu'$ be a weak* accumulation point of $(\mu_i)$. 
As before we have that
\begin{equation} \label{ineq: 15} \ 
\mu'(\Sigma \times E) \geq (k+1) \cdot \overline{\diff}(A)-k,
\end{equation}
On the other hand, for any $i\in \N$, $i\geq k+1$, we have that
\begin{equation}
\sum_{j=0}^{q} \mu_i(S^{-j}E \times \Sigma) = \frac{1}{N'_i} \sum_{j=0}^{q} \sum_{n=1}^{N'_i} \delta_{S^{(q+1)n}a'}(S^{-j}E) = \frac{k+1}{N_i} \left( \sum_{n=1}^{N_i'(q+1)} \delta_{S^{n}a'}(E)\right)  +o_{N_i\to \infty}(1). \label{ineq: 16}
\end{equation} 
We can relate this with the density of $A$ because 
\begin{align*}
\sum_{n=1}^{N_i'(q+1)} \delta_{S^{n}a'}(E) & = \sum_{n=1}^{\lfloor N_i/(\ell(k+1)) \rfloor} \sum_{j=0}^{\ell(k+1)-1} \delta_{S^{\ell(q+1)n+j}a'}(E) + \sum_{n=1}^{\lfloor N_i/(\ell(k+1)) \rfloor} \sum_{j=\ell(k+1)}^{\ell(q+1)-1} 1 + o_{N_i\to \infty}(N_i) \\ 
& = \sum_{n=1}^{\lfloor N_i/(\ell(k+1)) \rfloor} \sum_{j=0}^{\ell+m -1} \delta_{S^{(\ell+m )n+j}a}(E) +  \frac{N_i}{\ell(k+1)}\left( \ell(q+1)-\ell(k+1) \right) + o_{N_i\to \infty}(N_i) \\
& = \sum_{n=1}^{N_i} \delta_{S^n a} +N_i\frac{q-k}{k+1} + o_{N_i\to \infty}(N_i)\\ 
& = |A \cap [1, N_i ]| + N_i\frac{q-k}{k+1} + o_{N_i\to \infty}(N_i),
\end{align*}
where we have used, when convenient, the equality $\ell+m =\ell(k+1)$ and the fact that $q=\lceil k \rceil \geq k$. Using this in \eqref{ineq: 16} and taking limits as $i\to \infty$ we see that 
\begin{equation}\label{ineq: 17}
\sum_{j=0}^{q} \mu'(S^{-j}E \times \Sigma) \geq (k+1) \cdot \overline{\diff}(A)+ q-k.    
\end{equation}
It follows by combining \eqref{ineq: 15} and \eqref{ineq: 17} that 
$$(\ell+m)\mu'(\Sigma \times E) + \ell \sum_{j=0}^{q} \mu'(S^{-j}E \times \Sigma) \geq (\ell+m) \left((k+1)\cdot \overline{\diff}(A)-k \right)+\ell (k+1)\cdot \overline{\diff}(A)+\ell(q-k).$$
The rest of the proof follows the same arguments as the previous 
one.
\end{proof}

\section{Proofs of main dynamical theorems} \label{main proofs}

For the proof of our main result we need a lemma 
that guarantees $(T^{\ell} \times T^m)$-generic 
points, for almost all the measures 
$\lambda_{(a,x_1)}$ with prescribed first
coordinate.

\begin{lemma} \label{full measure generic points}
If $(X,\mu,T)$ is an ergodic system and $a\in \textbf{gen}(\mu,\Phi)$ for some F\o lner sequence $\Phi$, then for $\mu$-almost every $x_1 \in X$ we have that $(a,x_1)$ is $(T^{\ell} \times T^m)$-generic for $\lambda_{(a,x_1)}$.   
\end{lemma}

\begin{proof}
The proof is an adaptation of the proof of Lemma 
$3.12$ of \cite{kmrr2}, but requires a few adjustments which 
are not immediately obvious. 

Using property $(iii)$ of Proposition \ref{erg decomposition} and Fubini's theorem it follows 
that for each one of a full measure set of points 
$b\in \supp(\mu)$ there is a full measure set of points $x\in X$ so that $(b,x) \in \textbf{gen}(\lambda_{(b,x)},T^{\ell}\times T^m, \Phi')$, where $\Phi'=(\{1,\dots,N\})_{N\in \N}$. We let $(G_j)_{j=1}^{\infty}$ be a dense subset of $C(X \times X)$ and for each $j\in \N$ we set $\Tilde{G}_j(x,y)=\int_{X\times X} G_j \ d\lambda_{(x,y)}$. As the map $(x_1,x_2) \mapsto \lambda_{(x_1,x_2)}$ is continuous and $(T^{\ell} \times T^m)$-invariant, it follows that each $\Tilde{G}_j$ is continuous and $(T^{\ell} \times T^m)$-invariant.  

Now, as $a\in \textbf{gen}(\mu,T,\Phi)$ we have that $\supp(\mu) \subset \overline{O_T(a)} \subset \bigcup_{j=0}^{\ell-1} \overline{O_{T^{\ell}}(T^{j}a)}$, where $O_S(a)$ is the (forward) orbit of $a$ under a homeomorphism $S: X\to X$, i.e. $\{S^n a \colon n\in \N \}$. Therefore, $\mu(\overline{O_{T^{\ell}}(a)})>0$ and so we can choose some $b$ as above with $b\in \overline{O_{T^{\ell}}(a)}$. This allows us to find, for each $k\in \N$ and $N(k)\in \N$, an $s(k)\in \N$ such that $\max_{1\leq n \leq N(k)}\norm{G_j(T^{\ell n}b, \cdot) - G_j(T^{\ell(n+s(k))}a, \cdot)}_{\infty}< 2^{-k}$ and $\norm{\Tilde{G_j}(b, \cdot) - \Tilde{G_j}(T^{\ell s(k)}a, \cdot)}_{\infty}< 2^{-k}$ for all $j\leq k$. 

Because $(b,x)$ is $(T^{\ell} \times T^m)$-generic 
for $\lambda_{(b,x)}$ for $\mu$-almost every 
$x\in X$, we can find for each $k\in \N$ some $N(k)\in \N$ such that 
$$F_k(x)= \max_{1\leq j\leq k} \left| \frac{1}{N(k)} \sum_{n=1}^{N(k)} G_j(T^{\ell n} b, T^{m n} x) - \Tilde{G_j}( b, x) \right| $$
satisfies $\norm{F_k}_{L^1(\mu)} < 2^{-k}.$ By the choice of $s(k)$ we have that 
$$\Tilde{F_k}(x):= \max_{1\leq j\leq k} \left| \frac{1}{N(k)} \sum_{n=1}^{N(k)} G_j(T^{\ell (n+s(k))} a, T^{m n} x) - \Tilde{G_j}( T^{\ell s(k)}a, x) \right| $$
satisfies $\norm{\Tilde{F_k}}_{L^1(\mu)} < 3\big/ 2^{k}$. 
Then, by the $(T^{\ell} \times T^m)$-invariance of
$\Tilde{G}_j$, we can consider $\Psi_k=\{s(k)+1,\dots,s(k)+N(k)\}$ and then we see that 
$$F'_k(x):=\Tilde{F_k}(T^{m s(k)} x)= \max_{1\leq j\leq k} \left| \frac{1}{|\Psi_k|} \sum_{n\in \Psi_k} G_j(T^{\ell n} a, T^{m n} x) - \Tilde{G_j}( a, x) \right|.$$ As $\mu$ is $T-$invariant it follows that $\norm{F'_k}_{L^1(\mu)}=\norm{\Tilde{F}_k}_{L^1(\mu)}$ and so, if we define $F(x):=\sum_{k=1}^{\infty} F'_k(x)$ it follows that $\norm{F}_{L^1(\mu)} < \infty$ and so $F$ is finite $\mu$-almost everywhere. Finally, for each $x_1\in X$ such that $F(x_1)<\infty$ it must hold
that 
$F'_k(x_1) \xrightarrow{} 0$ as $k\to \infty$ and 
so $(a,x_1) \in \textbf{gen}(\lambda_{(a,x_1)},T^{\ell}\times T^m, \Psi)$. 
\end{proof}

Combining some of the above results we can 
guarantee the existence of $(\ell,m)$-Erd\H{o}s 
progressions via the following analogue of 
Proposition $3.3$ in \cite{kousek_radic2024}. The proof is 
almost identical to that of the latter mentioned
proposition and so we omit it. For reference, the established properties of the measures $\sigma_a$ and $\lambda_{(x_1,x_2)}$ used in the proof are those in \cref{full measure generic points}, \cref{projections of sigma} and \cref{support of sigma}.

\begin{proposition}\label{EP full measure}
Let $(X,\mu,T)$ be an ergodic system and assume there is a 
continuous factor map $\pi\colon  X \to Z$ to its Kronecker factor. 
Let 
$a\in \textbf{gen}(\mu,\Phi)$, for some F\o lner sequence $\Phi$. 
Then for $\sigma_a$-almost every $(x_1,x_2) \in X \times X$, the 
point 
$(a,x_1,x_2)$ is an $(\ell,m)$-Erd\H{o}s progression. 
\end{proposition}

We are now in the position to prove all our main dynamical 
results. Before presenting the proofs, we emphasise again that, 
according to our discussion in the beginning of Section \ref{cts ergodic decomposition section}, these theorems also imply Theorems \ref{dynamical mB+lB}, \ref{dynamical mB+lB unrestricted shift} and \ref{dynamical mB+lB unrestricted no shift}, respectively.

\begin{proof}[Proof of \cref{continuous dynamical mB+lB}]
Since for $\sigma_a$-almost every point
$(x_1,x_2)\in X\times X$ the triple 
$(a,x_1,x_2)$ is an $(\ell,m)$-Erd\H{o}s 
progression by Proposition \ref{EP full measure}, we simply need to show that $\sigma_a(X \times T^{-t}E)>0$ for some $t\in \N.$ The latter 
follows directly by the fact that 
$\pi_2 \sigma_a \left( \bigcup_{t\in \N} T^{-t}E 
\right) >0$. Indeed, as $\mu$ is ergodic we 
have that 
$\mu(\bigcup_{t\in \N}T^{-t}E)= 1$ and thus, in 
way of Lemma \ref{projections of sigma}, we 
actually see that $\pi_2 \sigma_a \left( \bigcup_{t\in \N} T^{-t}E 
\right) = 1$.
\end{proof}

\begin{remark}\label{lB in shift}
In fact one can modify the above proof to show that there 
exist $t_1,t_2\in \N$ such that $\sigma_a(E\times T^{-t_1}E)>0$ and $\sigma_a(T^{-t_2}E \times T^{-t_2}E)>0.$ These in turn guarantee $(\ell,m)$-\Erdos\ progressions $(a,x_1,x_2), (a,x_1',x_2')\in X^3$ with $(x_1,T^{t_1}x_2),(T^{t_2}x_1',T^{t_2}x_2') \in E\times E$ and 
then, \cref{mB+lB} can be strengthened, so that for any set $A\subset \N$ with $\diff^{*}(A)>0$ the following hold: 
\begin{enumerate}[(i)]
    \item There exist an infinite set $B\subset \N$ and a shift $t\in \N$ such that $\ell B \subset A$ and
    $$\{mb_1 + \ell b_2\colon  b_1,b_2\in B,\ b_1 < b_2\} +t \subset A.$$  
    \item There exist an infinite set $B\subset \N$ and a shift $t\in \N$ such that 
    $$\ell B \cup \{mb_1 + \ell b_2\colon  b_1,b_2\in B,\ b_1 < b_2\} \subset A-t.$$  
\end{enumerate}

\end{remark}

\begin{proof}[Proof of \cref{continuous dynamical mB+lB unrestricted shift}]
Applying Proposition \ref{EP full measure} once again, we only have 
to show that for some $j\in \{1,\ldots,\ell+m\}$, $\sigma_a(E_j \times F_j)>0$, because this would imply the existence of
$x_1,x_2 \in X$ so that $(a,x_1,x_2) \in X^3$ is an $(\ell,m)$-Erd\H{o}s progression 
and $(x_1,x_2) \in E_j \times F_j$. This would follow from the inequality
\begin{equation} \label{ineq:1}  
\sigma_a(E_j\times X)+ \sigma_a(X\times F_j)>1,    
\end{equation}
because $$E_j \times F_j=\left(E_j\times X\right) \cap \left(X\times F_j\right)$$ and $\sigma_a$ is a probability measure. It suffices to show that \eqref{ineq:1} has to be satisfied with $j=1$ provided that it fails for $j=2,\ldots,\ell+m$. In particular, we assume that 
$$\sum_{j=2}^{\ell+m} \sigma_a(E_j\times X)+ \sigma_a(X\times F_j) \leq \ell+m-1,$$
which can be rewritten as 
\begin{equation} \label{ineq:2}
\sum_{j=2}^{\ell+m} \pi_1\sigma_a(E_j)+ \pi_2\sigma_a(F_j) \leq \ell+m-1.    
\end{equation}
Under this assumption, we have to show that $\pi_1\sigma_a(E_1)+\pi_2 \sigma_a(F_1)>1$, which by \cref{projections of sigma} can be rewritten as 
\begin{equation}\label{ineq:3}
\pi_1\sigma_a(E_1)+(\ell+m)\mu(F_1)-\sum_{j=1}^{\ell+m-1}T^{j}\pi_2\sigma_a(F_1)>1.    
\end{equation}
Recalling that $F_j=T^{-(j-1)}F_1$, for each $j=1,\ldots,\ell+m$, we see that \eqref{ineq:3} is equivalent to 
\begin{equation*}
\pi_1\sigma_a(E_1)+(\ell+m)\mu(F_1)-\sum_{j=2}^{\ell+m}\pi_2\sigma_a(F_j)>1,    
\end{equation*}
which by \eqref{ineq:2} would follow from 
\begin{equation} \label{ineq:4}
\pi_1\sigma_a(E_1)+(\ell+m)\mu(F_1)+\sum_{j=2}^{\ell+m}\pi_1\sigma_a(E_j)-(\ell+m-1)>1.   
\end{equation}
Next, we note that \eqref{ineq:4} can be rewritten as
\begin{equation} \label{ineq:4'}
(\ell+m)\mu(F_1)+\sum_{j=1}^{\ell+m}\pi_1\sigma_a(E_j)>\ell+m. \end{equation}
We now recall that $E_{i+(q+1)}=T^{-1}E_i$, for $i=1,\ldots,\ell+m-q-1$, and if $q>m/\ell$ (i.e. whenever $m/\ell \in \Q$) we also consider auxiliary sets $E_{\ell+m+1},\ldots,E_{\ell(q+1)}$ such that $E_{i+(q+1)}=T^{-1}E_i$, for $i=1,\ldots,(\ell-1)(q-1)$. Then, 
$$\sum_{j=1}^{\ell(q+1)}\pi_1\sigma_a(E_j)=\sum_{j=1}^{q+1} \sum_{i=0}^{\ell-1} T^{i}\pi_1\sigma_a(E_j)=\ell \sum_{j=1}^{q+1} \mu(E_j) $$
in accordance with \cref{projections of sigma}. We also make the trivial observation that 
$$ \sum_{j=1}^{\ell(q+1)}\pi_1\sigma_a(E_j) -\sum_{j=1}^{\ell+m}\pi_1\sigma_a(E_j) = \sum_{j=\ell(k+1)+1}^{\ell(q+1)} \pi_1\sigma_a(E_j) \leq \ell(q-k),$$
since $\ell(k+1)=\ell+m$. Therefore, it follows that 
$$(\ell+m)\mu(F_1)+\sum_{j=1}^{\ell+m}\pi_1\sigma_a(E_j) \geq (\ell+m)\mu(F_1)+\ell \sum_{j=1}^{q+1}\mu(E_j)-\ell(q-k)$$
and so \eqref{ineq:4'} follows from \eqref{cts mB+lB shift eq}.
\end{proof}

\begin{proof}[Proof of \cref{continuous dynamical mB+lB unrestricted no shift}]
We need to find an Erd\H{o}s progression of the 
form 
$(a,x_1,x_2)\in X^3$  
with $(x_1,x_2) \in E \times F$ and as before, by Proposition 
\ref{EP full measure} it suffices to show that 
$\sigma_a(E \times F)>0$. Once again, this would follow from
\begin{equation*}   
\sigma_a(E\times X)+ \sigma_a(X\times F)>1.    
\end{equation*} 
To this end, we simply note that 
\begin{align*}
    \sigma_a(E \times X)+\sigma_a(X \times F)&=\pi_1\sigma_a(E)+\pi_2\sigma_a(F) & \\
    &= \ell \mu(E)+(\ell+m)\mu(F)-\sum_{j=1}^{\ell-1}\pi_1\sigma_a(T^{-j}E)- \sum_{i=1}^{\ell+m-1}\pi_2\sigma_a(T^{-i}F) & \quad \\
    &\geq \ell\mu(E)+(\ell+m)\mu(F)-(\ell-1)-(\ell+m-1)>1, & \quad  
\end{align*}
where the second equality follows by \cref{projections of sigma} and the strict inequality by the assumption in \eqref{cts mB+lB no shift eq}. This concludes the proof. 
\end{proof}

\section{Examples for optimality}\label{counterexamples}

In this short section we will show that Theorems 
\ref{mB+lB unrestricted shift} and \ref{mB+lB unrestricted no shift} are optimal. That is, the density thresholds presented in 
both these results cannot be improved.

\begin{proposition} \label{counterexamples upper density}
Let $\ell,m\in \N$ and $k=m/\ell$. Then, there exist two sets 
$A,A'\subset \N$ with $\overline{\diff}(A)=(k+1)/(k+2)=1-1/(k+2)$ and $\overline{\diff}(A')=1-1/\left(\ell(k+1)(k+2)\right)$ such that for any infinite set $B$ and any integer $t\in \N$ we have that $\{mb_1 + \ell b_2 : b_1,b_2 \in B\ \text{and}\ b_1 \leq b_2\}+t \not\subset A$ and $\{mb_1 + \ell b_2 : b_1,b_2 \in B\ \text{and}\ b_1 \leq b_2\} \not\subset A'.$   
\end{proposition}

\begin{proof}[Proof of \cref{counterexamples upper density}]
Let $A$ be the subset of $\N$ defined by
\begin{equation*}
    A = \N \cap \bigcup_{n \in \N} [(k+1)^{2n}, (k+1-1/n) \cdot (k+1)^{2n}).
\end{equation*}
It is clear by the definition of $A$ that $\overline{\diff}(A) = 
\diff_{\Phi} (A)$, where $\Phi=(\Phi_N)$ is the F\o lner sequence 
given by $N \mapsto  [1, (k+1-1/N) (k+1)^{2N}) \cap \N$. We show 
that $\overline{\diff}(A)=(k+1)/(k+2).$ 

Indeed, first observe that $\diff_{\Phi} (A)=\diff_{\Tilde{\Phi}} (\Tilde{A})$, where 
\begin{equation*}
    \Tilde{A} = \N \cap \bigcup_{n \in \N} [(k+1)^{2n}, (k+1)^{2n+1})
\end{equation*}
and $\Tilde{\Phi}=(\Tilde{\Phi}_N)$ is the sequence given by $N \mapsto  [1, (k+1)^{2N+1}) \cap \N$ and then 
\begin{multline*}
\frac{\left| \Tilde{A} \cap [1, (k+1)^{2N+1})] \right|}{ (k+1)^{2N+1}} = \frac{1}{(k+1)^{2N+1}} \sum_{n=1}^{N} k(k+1)^{2n} = \\ 
\frac{k}{k+1} \sum_{n=1}^N \frac{1}{((k+1)^{2})^{N-n}} \xrightarrow{N\to \infty} \frac{k}{k+1} \cdot \frac{1}{1-\frac{1}{(k+1)^2}}=\frac{k+1}{k+2}. 
\end{multline*}
Now, assume there exist an infinite $B\subset \N$ and some integer $t\in \N$ for which $\{mb_1 + \ell b_2 : b_1,b_2 \in B\ \text{and}\ b_1 \leq b_2\}+t \subset A$. In particular, for any $b'\in B$ fixed there is $b\in B$ arbitrarily large such that $\{(\ell+m )b+t,\ell b+mb'+t\}\subset A$. Then, we can choose $b$ so that, $(\ell+m )b+t \in [(k+1)^{2(n+1)}, (k+1-\frac{1}{n+1}) (k+1)^{2(n+1)}))$, for some $n \in \N$ with respect to which both $t$ and $b'$ are negligible. Note that $\ell+m=\ell(k+1)$ and so it follows that
$$ (k+1)^{2n+1} < \ell b + t < \left(k+1-\frac{1}{n+1}\right) (k+1)^{2n+1} + t. $$
By the choice of $b$ with respect to $t$ and $b'$ we see that 
$mb'+\ell b+t \in [(k+1)^{2n+1},\ (k+1)^{2(n+1)}) \subset \N \setminus A$, reaching a 
contradiction. To see why the last claim is true, observe that $$(k+1)^{2(n+1)}-\left(k+1-\frac{1}{n+1}\right) (k+1)^{2n+1} = (k+1)^{2n+1} \cdot \frac{1}{n+1} \xrightarrow{n\to \infty} \infty.$$
This completes the first construction. 

Keeping $A$ as defined above we now consider
$A'=A \cup \left( \bigcup_{j=1}^{\ell+m -1}(\ell+m )\N+j \right)$. In other words, $A'=A \cup \left( \left(\N \setminus A \right) \setminus \left( (\ell+m )\N \right) \right)$. By the definition of $A$ we have that $\N \setminus A$ 
is a union of discrete intervals with lower density equal to $\underline{\diff}(\N \setminus A)=1-\overline{\diff}(A)=1/(k+2)$ and so 
$\underline{\diff}(\left(\N \setminus A \right) \cap (\ell+m ) \N )=1/(\ell+m )(k+2)=1/\left(\ell(k+1)(k+2)\right)$. Since the complement of $A'$ is precisely $\left( \N \setminus A \right) \cap  (\ell+m )\N $, it follows that $\overline{\diff}(A')=1-1/\left(\ell(k+1)(k+2)\right)$. 
Finally, we claim there is no infinite set $B\subset \N$ satisfying $\{mb_1 + \ell b_2 : b_1,b_2 \in B\ \text{and}\ b_1 \leq b_2\} \subset A'$. Indeed, if there were such a set, we could consider 
an infinite subset $B'\subset B$ consisting of integers which are equal modulo $(\ell+m)$ (see also the proof of \cref{mB+lB no shift restricted}) 
and so we would have that $\{mb_1 + \ell b_2 : b_1,b_2 \in B\ \text{and}\ b_1 \leq b_2\} \subset A' \cap (\ell+m )\N \subset A$. This contradicts the first construction and so we conclude.  
\end{proof}

\section{Lower density results}\label{lower density section}

We also want to briefly explore what further results we can get 
by considering the input of the information provided by a set's 
lower natural density as well. The proofs in this section are 
straightforward adaptations of the arguments used thus far, 
so we omit repetitive details, but we include comments 
regarding all the non-obvious changes 
in the argumentation. 

A possibility to capture this new input -- one that has been tried 
and found to be fruitful in \cite{kousek_radic2024} -- is to 
replace \eqref{ineq: 12}, which appears in both
correspondence principles in Lemmas \ref{FCP shift general k} and \ref{FCP no shift general k}, by the always true inequality $\mu'(\Sigma \times E) \geq \underline{\diff}(A).$
To see this recall that $\mu'$ is defined as a weak* limit of 
the measure sequence $(\mu_i)$ and take the liminf as 
$i\to\infty$ in \eqref{ineq: 11}. Making this simple change in the proofs of the correspondence principles we recover the following 
results.

\begin{lemma} \label{FCP lower density shift} 
 Let $A\subset \N$ and $\ell,m \in \N$ with $k=m/\ell$ and let $q=\lceil k \rceil$, the ceiling of $k$. 
 Then, there exist an ergodic system $(\Sigma 
 \times \Sigma, \mu, S^{(q+1)} \times S)$, an open set $E\subset \Sigma$, a pair of points $a,a'\in \Sigma$ and a F\o lner sequence 
 $\Phi$, such that $(a',a)\in \textbf{gen}(\mu,\Phi)$ and 
 \begin{equation*} (\ell+m)\mu(\Sigma \times E) + \ell \sum_{j=0}^{q} \mu(S^{-j}E \times \Sigma) \geq (\ell+m)\cdot \underline{\diff}(A)+\ell(k+1)\cdot \overline{\diff}(A) + \ell(q-k).
 \end{equation*}
 It also holds that $A=\{n\in \N: S^n a \in E \}$ and $(A-j)/(\ell+m ) =\{n\in \N: S^{(q+1)\ell n+j}a' \in E\}$, for each $j=0,1,\dots,\ell+m -1$, where $(A-j)/(\ell+m )=\{n\in \N: n(\ell+m )+j\in A\}$.
    
\end{lemma}

\begin{lemma} \label{FCP lower density no shift}
Let $A\subset \N$ and $\ell,m\in \N$ with $k=m/\ell$. Then, there exists an ergodic system $(\Sigma 
\times \Sigma, \mu, S \times S)$, an open set $E\subset \Sigma$, a pair of points $a'',a\in \Sigma$ and a F\o lner sequence 
$\Phi$, such that $(a'',a)\in \textbf{gen}(\mu,\Phi)$ and 
\begin{equation*} 
(\ell+m)\mu(\Sigma \times E) + \ell \mu(E \times \Sigma) \geq  (\ell+m)\cdot \underline{\diff}(A) +\ell \left( (k+1)\cdot \overline{\diff}(A)-k \right).
\end{equation*}
It also holds that $A = \{ n \in \N: S^n a \in E \}$ and $A/(\ell+m ) =\{n\in \N: S^{\ell n}a'' \in E\}$. 
\end{lemma}

Now that we have these additional correspondence principles we 
may expand on the arguments presented in Section \ref{unrestricted results}, using Theorems \ref{dynamical mB+lB unrestricted shift}, \ref{dynamical mB+lB unrestricted no shift} and Proposition \ref{EP and return times} to get the following combinatorial result. 

\begin{theorem} \label{mB+lB lower density}
Let $A\subset \N$ and $\ell,m \in \N$.
\begin{enumerate}
    \item If $\underline{\diff}(A)+\overline{\diff}(A)>1,$ there is an infinite set $B\subset \N$ and some $t\in \{0,1,\ldots,\ell+m-1\}$ such that $\{mb_1 + \ell b_2\colon  b_1,b_2\in B,\ b_1 \leq b_2\}+t \subset A.$
    \item If $\underline{\diff}(A)+\overline{\diff}(A)>2-1/(\ell+m),$ there is an infinite set $B\subset \N$ such that $\{mb_1 + \ell b_2\colon  b_1,b_2\in B,\ b_1 \leq b_2\} \subset A.$
\end{enumerate}
\end{theorem}

We have the following immediate corollary of \cref{mB+lB lower density}.

\begin{corollary} \label{mB+lB lower density cor.}
Let $A\subset \N$ and $\ell,m \in \N$. 
\begin{enumerate}
    \item If $\underline{\diff}(A)>1/2,$ there is an infinite set $B\subset \N$ and some $t\in \{0,1,\ldots,\ell+m-1\}$ such that $\{mb_1 + \ell b_2\colon  b_1,b_2\in B,\ b_1 \leq b_2\}+t \subset A.$
    \item If $\underline{\diff}(A)>1-1/(2(\ell+m)),$ there is an infinite set $B\subset \N$ such that $\{mb_1 + \ell b_2\colon  b_1,b_2\in B,\ b_1 \leq b_2\} \subset A.$
\end{enumerate}
\end{corollary}

\begin{remark*}
Compare this result with the upper density analogues in \ref{mB+lB unrestricted shift} and \ref{mB+lB unrestricted no shift}. In particular, the threshold values for upper density depend on the parameters $m$ and $\ell$ involved in the sumsets. It is very surprising that this is no longer the case for threshold values of lower density, at least in the case of shifted sumsets. Of course, some kind of dependence in the case of unshifted patterns is enforced by the fact that infinite sumsets $mB+\ell B$ essentially ``live'' in $(m+\ell)\N$, as was utilised in the proof of \cref{mB+lB no shift restricted}.
\end{remark*}

We will conclude this discussion by proving that Corollary 
\ref{mB+lB lower density cor.} (and thus \cref{mB+lB lower density} too) is also optimal. 

\begin{proposition}\label{lower density examples}
Let $\ell,m\in \N$. There exist two sets 
$A,A'\subset \N$ with $\underline{\diff}(A)=1/2$ and $\underline{\diff}(A')=1-1/(2(\ell+m))$ such that for any infinite set $B$ and any integer $t\in \N$ we have that $\{mb_1 + \ell b_2 : b_1,b_2 \in B\ \text{and}\ b_1 \leq b_2\}+t \not\subset A$ and $\{mb_1 + \ell b_2 : b_1,b_2 \in B\ \text{and}\ b_1 \leq b_2\} \not\subset A'.$       
\end{proposition}

\begin{proof}[Proof of \cref{lower density examples}]
Let $k=m/\ell$ and $A_1,A_2$ be the subsets of $\N$ defined by
\begin{equation*}
    A_1 = \N \cap \bigcup_{n \in \N} [(k+1)^{2n}, (k+1-1/n) \cdot (k+1)^{2n})
\end{equation*}
and 
\begin{equation*}
    A_2 = \N \cap \bigcup_{n \in \N} [ (k+1+1/n) \cdot (k+1)^{2n}, (k+1)^{2(n+1)}).
\end{equation*}
We saw in \cref{counterexamples upper density} that there is no 
infinite set $B\subset \N$ and integer $t\in \N$ so that $\{mb_1 + \ell b_2 : b_1,b_2 \in B\ \text{and}\ b_1 \leq b_2\}+t \subset A_1$ and a symmetrical argument shows the same conclusion holds for $A_2$ as well.

Next, we observe that the set 
$$\N \setminus \left( A_1 \cup A_2 \right) = \N \cap \bigcup_{n \in \N} \left[ (k+1-1/n) \cdot 
(k+1)^{2n}, (k+1+1/n) \cdot (k+1)^{2n} \right]$$ 
has zero density. This follows from the fact that 
$$\lim_{N\to \infty} \sum_{n=1}^N \frac{1}{n \cdot (k+1)^{2(N-n)}}= 0,$$
which holds because $k+1>1$ (for a similar argument see the 
proof of Proposition $4.2$ in \cite{kousek_radic2024}).

We define $A\subset \N$ by
$$A = \left(A_1 \cap 2\N \right) \cup \left( A_2 \cap (2\N+1) \right)$$
and claim that this set satisfies the properties in the 
statement above. Clearly, by the last argument, we see that 
$\diff(A)=1/2$ and so, in particular, $\underline{\diff}(A)=1/2$. For the second property, assume for contradiction the existence of an infinite set $B\subset$ and some integer $t\in \N$ so that $\{mb_1 + \ell b_2 : b_1,b_2 \in B\ \text{and}\ b_1 \leq b_2\}+t \subset A$. We can pass to an infinite subset $B'\subset B$ all the elements of which have the same parity. 
Then, there is $i\in \{0,1\}$ such that for each 
$b_1,b_2 \in B'$ there are $n_1,n_2\in \N$ for which 
$b_1=2n_1+i$, $b_2=2n_2+i$ and then $mb_1+\ell b_2+t=2(mn_1+\ell n_2)+(\ell+m )i+t \in 2\N + (\ell+m )i+t$. But then, depending on wether $(\ell+m )i+t$ is even or odd we have that $\{mb_1 + \ell b_2 : b_1,b_2 \in B'\ \text{and}\ b_1 \leq b_2\}+t$ is contained in $A\cap 2\N$ or $A\cap (2\N+1)$, respectively. In each case we reach a contradiction, because it would have either have to hold that $\{mb_1 + \ell b_2 : b_1,b_2 \in B'\ \text{and}\ b_1 \leq b_2\}+t \subset A_1$ or $\{mb_1 + \ell b_2 : b_1,b_2 \in B'\ \text{and}\ b_1 \leq b_2\}+t \subset A_2$.

For the second part of the construction we keep $A_1$ and $A_2$ 
as they were defined above. Then, we let $A'\subset \N$ be 
defined by 
$$A'= \left(A_1 \setminus (2(\ell+m )\N \right) \cup 
\left( A_2 \setminus (2(\ell+m )\N+(\ell+m )) \right).$$
We already showed that $\diff(A_1 \cup A_2)=1$ and it is easy to 
see that $\diff(A')=1-1/(2(\ell+m ))$, because we are only 
removing a set of density $1/(2(\ell+m ))$ from $A_1 \cup A_2$. 
Finally, 
we claim there is no infinite set $B\subset \N$ satisfying 
$\{mb_1 + \ell b_2 : b_1,b_2 \in B\ \text{and}\ b_1 \leq b_2\}\subset A'$. Indeed, given such an infinite set $B$ we can
consider an infinite subset $B' \subset B$ with all its 
elements equivalent modulo $2(\ell+m )$. That is, there exists 
$j\in \{0,1,\dots,2(\ell+m )-1\}$ so that $b'\equiv j \pmod {2(\ell+m )}$, for any $b'\in B'$. But then we have 
$\{mb_1 + \ell b_2 : b_1,b_2 \in B\ \text{and}\ b_1 \leq b_2\}\subset 2(\ell+m )\N+(\ell+m )j.$
Depending on the parity of $j$ this either implies that $\{mb_1 + \ell b_2 : b_1,b_2 \in B\ \text{and}\ b_1 \leq b_2\}\subset A_1$ or $\{mb_1 + \ell b_2 : b_1,b_2 \in B\ \text{and}\ b_1 \leq b_2\}\subset A_2$, both of which are contradictions to the first part of the construction. 
\end{proof}

\begin{remark}
We stress that for both sets $A,A'$ above we actually have that 
their natural densities are realised. This was to be expected 
because of \cref{mB+lB lower density}. Indeed, we could not have, for example, 
a set $A\subset \N$ with $\underline{\diff}(A)=1/2$ but $\overline{\diff}(A)>1/2$ (i.e. the natural density of $A$ not realised) and also such that $\{mb_1 + \ell b_2 : b_1,b_2 \in B\ \text{and}\ b_1 \leq b_2\}+t \not\subset A$ for any infinite set $B$ and any integer $t\in \N$, because this would violate part $(1)$ of \cref{mB+lB lower density}.
\end{remark}

\section{Remarks and questions about further extensions}\label{questions}

A significant part of this work was focused on extensions
of \cref{mB+lB} by relaxing the restrictions imposed on the 
sumsets involved. In short, we managed to provide complete 
characterisation for the existence of the patterns $\{mb_1 + \ell b_2: b_1,b_2\in B\ \text{and}\ b_1\leq b_2 \}$ for infinite $B\subset \N$ in shifts of a set $A\subset \N$, based on the values for the upper and lower natural densities of $A$. 

Another direction could be to consider patterns of the form $\{mb_1 + \ell b_2: b_1,b_2\in B\ \text{and}\ b_1\neq b_2 \}$. We recall \cref{unrestricted patterns counterexample}, according to which an assumption of positive upper density is not sufficient to find shifts of such infinite sumsets in a subset of the integers (unless, of course, $\ell=m$). A natural question, then, is if this problem has a `density solution'. We present two related constructions 
(which, we believe, can not be improved) and then state the
question precisely. 

\begin{proposition}\label{more examples}
Let $\ell,m \in \N$ be distinct with $\ell>m$. There exist two sets $A, A'\subset \N$ with $\overline{\diff}(A)=\ell/(\ell+m)$ and $\overline{\diff}(A)=1-m/(\ell+m)^2$ such that for any infinite set $B\subset \N$ and any $t\in \N$ we have $\{mb_1 + \ell b_2: b_1,b_2\in B\ \text{and}\ b_1\neq b_2 \} \not \subset A-t$ and $\{mb_1 + \ell b_2: b_1,b_2\in B\ \text{and}\ b_1\neq b_2 \} \not \subset A'$.  
\end{proposition}

\begin{remark}\label{golden ratio}
It may be interesting to compare these bounds with the optimal bounds established for the sumsets $\{mb_1 + \ell b_2 \colon b_1,b_2 \in B,\ \text{and}\ b_1\leq b_2\}$ in Theorems \ref{mB+lB unrestricted shift}, \ref{mB+lB unrestricted no shift} and Proposition \ref{counterexamples upper density}. In particular, by assumption we have that $k=m/\ell<1$ and then $\ell/(\ell+m)=1/(k+1)$, and this is greater than $(k+1)/(k+2)$, whenever $k(k+1)<1$. Also, $1-m/(\ell+m)^2$ is greater than $1-1/(\ell(k+1)(k+2))=1-1/((\ell+m)(k+2))$, again, precisely when $k(k+1) < 1$. 

The surprising fact that springs from this observation is the
following; Say, without loss of generality, that $\ell>m$. Then, the largest bounds (we can find) for the value of upper density for a set that doesn't contain sumsets of the form $\{mb_1 + \ell b_2 \colon b_1,b_2 \in B,\ \text{and}\ b_1 \neq b_2\}$ are greater than the optimal bounds for a set that doesn't contain sumsets $\{mb_1 + \ell b_2 \colon b_1,b_2 \in B,\ \text{and}\ b_1 \leq b_2\}$ whenever $\ell/m=1/k$ is greater than the Golden ratio! The relation between the bounds is reversed if $\ell/m=1/k$ is less than the Golden ratio and the same comparisons hold for unshifted patterns. 

We can also compare the bounds of \cref{more examples} with the ones from the aforementioned results, but for the sumsets $\{\ell b_1+mb_2 \colon b_1,b_2\in B\ \text{and}\ b_1\leq b_2\}$. In this case, the comparison is not nearly as mystical, for the latter sumsets cannot be found in sets of upper density up to $(k^{-1}+1)/(k^{-1}+2)$ and $1-1/(m(k^{-1}+1)(k^{-1}+2))$, for the cases of shift and no shift, respectively. But both of these are larger than the respective bounds from \cref{more examples}, i.e. $\ell/(\ell+m)$ and $1-m/(\ell+m)^2$.
\end{remark}

\begin{proof}[Proof of \cref{more examples}]
The proof has similar features to the proofs of Propositions \ref{counterexamples upper density} and \ref{lower density examples}, so we try to ease exposition by avoiding repetitive arguments. 

We consider the 
set $A\subset \N$ defined by
$$A = \N \cap \bigcup_{n \in \N} [(\ell/m)^{2n}, (\ell/m-1/n) \cdot (\ell/m)^{2n})$$
Similarly to the proof of \cref{counterexamples upper density} we see that $\overline{\diff}(A)=\ell/(\ell+m)$. We claim that there is no infinite set $B\subset \N$ and integer $t\in \N$ such that $\{mb_1 + \ell b_2: b_1,b_2\in B\ \text{and}\ b_1\neq b_2 \} \not \subset A-t$. 

Indeed, say $B\subset \N$ is an infinite set and $t\in \N$ 
an integer negating the claim. Then, for any $b_1\in B$ we may choose $b_2\in B$ arbitrarily large with respect to $b_1$ and note that $\alpha=mb_1+\ell b_2$, $\beta=mb_2 + \ell b_1  \in A-t$. We observe that $\alpha=(\ell/m)\beta - (\ell^2/m-m)b_1$ and so we can choose $b_2$ so large that $c(b_1):=(\ell^2/m-m)b_1$ is negligible (in a way to be made precise below) with respect to $\alpha$. 

Let $n\in \N$ be such that $\alpha \in  [(\ell/m)^{2n}, (\ell/m-1/n) \cdot (\ell/m)^{2n})-t$. Then, 
$$\beta = \frac{1}{(\ell/m)}\left(\alpha+c(b_1)\right) \in [(\ell/m)^{2n-1}, (\ell/m-1/n) \cdot (\ell/m)^{2n-1}) + \frac{(c(b_1)-t)}{(\ell/m)}$$
and we choose $b_2$ large enough so that $n\in \N$ is in turn large enough in order for the above to imply that
$\beta + t \in  [(\ell/m)^{2n-1}, (\ell/m)^{2n}),$
hence $\beta \notin A-t$, a contradiction. 

We finally consider the set $A'=A \cup \left( \bigcup_{j=1}^{\ell+m-1} (\ell+m)\N+j \right)$ and arguing as in the proof of \cref{counterexamples upper density} we see that $\overline{\diff}(A)=1-m/(\ell+m)^2$ and that there is no infinite set $B\subset \N$ such that $\{mb_1 + \ell b_2: b_1,b_2\in B\ \text{and}\ b_1\neq b_2 \} \subset A'.$ 
\end{proof}

\begin{remark*}
Observe that the set $A\subset \N$ above is thick and as 
such ${\diff}^{*}(A)=1$.     
\end{remark*}

Despite the urge to make a conjecture, due to limited dynamical as well as combinatorial evidence -- apart from the very interesting \cref{golden ratio}, which hints that the bounds of \cref{more examples} may not be unrelated to the optimal bounds -- we constrain ourselves to asking the following question. 

\begin{question}\label{Q1}
Let $\ell,m \in \N$ be distinct with $\ell>m$ and let $A\subset \N$. 
\begin{enumerate}[(i)]
    \item If $\overline{\diff}(A)>\ell/(\ell+m)$, does there exist an infinite set $B\subset \N$ and some $t\in \N$ such that $\{mb_1 + \ell b_2: b_1,b_2\in B\ \text{and}\ b_1\neq b_2 \}+t \subset A$?  
    \item If $\overline{\diff}(A)>1-m/(\ell+m)^2$, does there exist an infinite set $B\subset \N$ such that $\{mb_1 + \ell b_2: b_1,b_2\in B\ \text{and}\ b_1\neq b_2 \} \subset A$?  
\end{enumerate}
\end{question}

With the sets constructed in \cref{more examples} as a basis, we can mimic the argument from the proof of \cref{lower density examples} and recover the following result. 

\begin{proposition} \label{more examples lower density}
Let $\ell,m\in \N$ be distinct with $\ell>m$. There exist two sets $A,A'\subset \N$ with $\underline{\diff}(A)=1/2$ and $\underline{\diff}(A')=1-1/(2(\ell+m))$ such that for any infinite set $B$ and any integer $t\in \N$ we have that $\{mb_1 + \ell b_2 : b_1,b_2 \in B\ \text{and}\ b_1 \neq b_2\}+t \not\subset A$ and $\{mb_1 + \ell b_2 : b_1,b_2 \in B\ \text{and}\ b_1 \neq b_2\} \not\subset A'.$     
\end{proposition}

\begin{proof}
Let $A_1 , A_2 \subset \N$ be defined by
$$A_1 = \N \cap \bigcup_{n \in \N} [(\ell/m)^{2n}, (\ell/m-1/n) \cdot (\ell/m)^{2n})$$
and
$$A = \N \cap \bigcup_{n \in \N} [(\ell/m+1/n)\cdot (\ell/m)^{2n}, (\ell/m)^{2(n+1)}).$$
By the proof of \cref{more examples} and a symmetrical argument for $A_2$, neither of the sets contains a sumset of the form $\{mb_1 + \ell b_2 : b_1,b_2 \in B\ \text{and}\ b_1 \neq b_2\}+t$, for any infinite $B\subset \N$ and any $t\in \N$. Then, analogously to the proof of \cref{lower density examples}, we define $A=(A_1\cap 2\N) \cup (A_2\cap (2\N+1))$ and $A'=(A_1 \setminus 2(\ell+m)\N) \cup (A_2 \setminus (2(\ell+m)\N+(\ell+m)))$ and the claimed properties of $A,A'$ follow similarly to the proof of \cref{lower density examples}.
\end{proof}

Therefore, we also ask the following. 

\begin{question}
Let $\ell,m \in \N$ be distinct and let $A\subset \N$. 
\begin{enumerate}[(i)]
    \item If $\underline{\diff}(A)>1/2$, does there exist an infinite set $B\subset \N$ and some $t\in \N$ such that $\{mb_1 + \ell b_2: b_1,b_2\in B\ \text{and}\ b_1\neq b_2 \}+t \subset A$?  
    \item If $\underline{\diff}(A)>1-1/(2(\ell+m))$, does there exist an infinite set $B\subset \N$ such that $\{mb_1 + \ell b_2: b_1,b_2\in B\ \text{and}\ b_1\neq b_2 \} \subset A$?  
\end{enumerate}
\end{question}

Provided Question \ref{Q1} has a positive answer one 
should also inquire about the completely unrestricted 
problem, that is, the existence of infinite sumsets 
$mB+\ell B=\{mb_1 + \ell b_2: b_1,b_2\in B\}$. To this end, we ask two more questions after making an 
observation which follows from current results. 

\begin{proposition}\label{counterexamples upper density unrestricted}
Let $\ell,m\in \N$ be distinct with $m>\ell$ and let $k=m/\ell$. Then, there exist sets $A,A' \subset \N$ with 
$\overline{\diff}(A)=(k+1)/(k+2)$
and $\overline{\diff}(A')=1-1/\left(\ell(k+1)(k+2)\right)$
such that for any infinite set $B$ and any integer $t\in \N$ we have that $\{mb_1 + \ell b_2 : b_1,b_2 \in B\}+t \not\subset A$ and $\{mb_1 + \ell b_2 : b_1,b_2 \in B\} \not\subset A'.$       
\end{proposition}

\begin{proof}
All we need to observe is that both sumsets $ \{mb_1 + \ell b_2 : b_1,b_2 \in B\ \text{and}\ b_1 \leq b_2\}$ and $\{\ell b_1 + m b_2 : b_1,b_2 \in B\ \text{and}\ b_1 \leq b_2\}$ are contained in $\{mb_1 + \ell b_2 : b_1,b_2 \in B\} $, for any $B\subset \N$, and so we can reduce to \cref{counterexamples upper density}. Indeed, if a set fails to contain either of the former sumsets, then it will also fail to contain $\{mb_1 + \ell b_2 : b_1,b_2 \in B\ \text{and}\ b_1 \leq b_2\}$. Now, as $\ell>m$, we have that $$\max\{(k+1)/(k+2),(k^{-1}+1)/(k^{-1}+2)\} = (k+1)/(k+2),$$
and so it follows by \cref{counterexamples upper density} that there exists a set $A\subset \N$ with $\overline{\diff}(A)=(k+1)/(k=2)$ and such that for any infinite $B\subset \N$ and $t\in \N$ we have that $\{m b_1 + \ell b_2 : b_1,b_2 \in B\ \text{and}\ b_1 \leq b_2\}+t \not \subset A$. In particular, it follows that $\{mb_1 + \ell b_2 : b_1,b_2 \in B\} \not \subset A-t$. 
\end{proof}

Repeating the same observation for values of lower 
density and \cref{lower density examples} we have the following proposition. 

\begin{proposition}\label{lower density examples unrestricted}
Let $\ell,m\in \N$. There exist two sets 
$A,A'\subset \N$ with $\underline{\diff}(A)=1/2$ and $\underline{\diff}(A')=1-1/(2(\ell+m))$ such that for any infinite set $B$ and any integer $t\in \N$ we have that $\{mb_1 + \ell b_2 : b_1,b_2 \in B\} \not\subset A-t$ and $\{mb_1 + \ell b_2 : b_1,b_2 \in B\} \not\subset A'.$       
\end{proposition}

It would be surprising if either of the next questions 
had a negative answer, but apart from our inability 
to improve on the previous constructions there is no evidence suggesting the opposite.

\begin{question}
Let $\ell,m \in \N$ be distinct with $m>\ell$ and let $k=m/\ell$. Let also $A\subset \N$. 
\begin{enumerate}[(i)]
    \item If $\overline{\diff}(A)>(k+1)/(k+2)$, does there exist an infinite set $B\subset \N$ and some $t\in \N$ such that $\{mb_1 + \ell b_2: b_1,b_2\in B\}+t \subset A$?  
    \item If $\overline{\diff}(A)>1-1/\left(\ell(k+1)(k+2)\right)$, does there exist an infinite set $B\subset \N$ such that $\{mb_1 + \ell b_2: b_1,b_2\in B\} \subset A$?  
\end{enumerate}    
\end{question}

\begin{question}
Let $\ell,m \in \N$ be distinct and $A\subset \N$. 
\begin{enumerate}[(i)]
    \item If $\underline{\diff}(A)>1/2$ does there exist an infinite set $B\subset \N$ and some $t\in \N$ such that $\{mb_1 + \ell b_2: b_1,b_2\in B \}+t \subset A$?  
    \item If $\underline{\diff}(A)>1-1/(2(\ell+m))$ does there exist an infinite set $B\subset \N$ such that $\{mb_1 + \ell b_2: b_1,b_2\in B\} \subset A$?  
\end{enumerate}    
\end{question}

As the above exposition on unrestricted sumsets $mB+\ell B$ focused on the case of distinct $\ell$ and $m$, we want to highlight that, since $\{mb_1 + mb_2 \colon b_1,b_2 \in B\ \text{and}\ b_1\leq b_2 \}=mB+mB$, the special unrestricted sumsets of the form $mB+m B$ are already covered in Theorems \ref{mB+lB unrestricted shift}, \ref{mB+lB unrestricted no shift}, \cref{mB+lB lower density cor.} and Propositions  \ref{counterexamples upper density}, \ref{lower density examples}. In particular, for some $A\subset \N$ and $m\in \N$ we have the following implications, all of which are optimal. 
\begin{enumerate}[(i)]
    \item If $\overline{\diff}(A)>2/3$, there exist some infinite set $B\subset \N$ and some $t\in \{0,1,\ldots,2m-1\}$ such that $mB+mB+t \subset A$. 
    \item If $\overline{\diff}(A)>1-1/(6m)$, there exists some infinite set $B\subset \N$ such that $mB+mB \subset A$. 
    \item If $\underline{\diff}(A)>1/2$, there exist some infinite set $B\subset \N$ and some $t\in \{0,1,\ldots,2m-1\}$ such that $mB+mB+t \subset A$. 
    \item If $\underline{\diff}(A)>1-1/(4m)$, there exists some infinite set $B\subset \N$ such that $mB+mB \subset A$. 
\end{enumerate}

We note that the bounds in $(i)$ do not depend on $m$, 
because the optimal bounds established in \cref{mB+lB unrestricted shift} explicitly depend only on the ratio $k=m/\ell$, which in this 
case is always $1$. This suggests that perhaps $(i)$ above 
already follows from \cref{KousekRadic theorem}, and this is 
indeed the case. The proof of this fact is nice and not very 
complicated, but serves no other purpose as to be included here. Similarly, using \cref{KousekRadic theorem}, one can prove that for any set $A\subset \N$ with $\overline{\diff}(A)>5/6$, there exists $B\subset \N$ infinite and some $t\in \{0,1,\ldots,m-1\}$ such that $mB+mB+t \subset A$ -- note the difference for the range of the shift here -- but of course, density larger than $5/6$ is no longer sufficient for us to remove potential shifts, as the bounds in part $(ii)$ above are optimal. 

\medskip

\bibliographystyle{abbrv}

\end{document}